\documentclass[12pt]{amsart}

\usepackage{geometry}            
\usepackage[applemac]{inputenc}
\usepackage[english]{babel}
\usepackage{enumerate}
\usepackage{amsmath, csquotes, bm, bbm, amssymb, mathrsfs, mathabx}
\usepackage{hyperref}
\usepackage{braket}

\newtheorem{theorem}{Theorem}[section]
\newtheorem{lemma}[theorem]{Lemma}
\newtheorem{proposition}[theorem]{Proposition}
\newtheorem{corollary}[theorem]{Corollary}
\theoremstyle{definition}
\newtheorem{notation}[theorem]{Notation}
\theoremstyle{definition}
\newtheorem{definition}[theorem]{Definition}
\theoremstyle{remark}
\newtheorem{remark}[theorem]{Remark}

\newtheorem*{remark*}{Remark}
\numberwithin{equation}{section}
\newenvironment{proof-sketch}{\textit{Sketch of proof.~}}{\quad}

\newcommand{\R}{\mathbb{R}}
\newcommand{\C}{\mathbb{C}}
\newcommand{\N}{\mathbb{N}}
\newcommand{\dx}{\mathrm{\,d}}
\newcommand{\n}{\mathbf{n}}
\newcommand{\x}{\mathbf{x}}

\newcommand{\y}{\mathbf{y}}
\newcommand{\z}{\mathbf{z}}
\newcommand{\pa}{\partial}
\newcommand{\eps}{\varepsilon}
\renewcommand{\Re}{\mathrm{Re}}

\DeclareMathOperator{\RE}{Re}
\newcommand{\Mat}{\mathsf{Mat}}
\newcommand{\jac}{\mathsf{Jac}}
\newcommand{\supp}{\mathsf{supp\,}}
\newcommand{\dive}{\mathrm{div}}
\newcommand{\D}{\mathsf{Dom}}

\title[On the MIT Bag Model]{On the MIT Bag Model:\\ Self-adjointness and Non-relativistic Limit}
\author{N. Arrizabalaga}
\address[N. Arrizabalaga]{Departamento de Matem\'aticas, Universidad del Pa\'is Vasco/Euskal Herriko Unibertsitatea (UPV/EHU), 48080 Bilbao, Spain}
\email{naiara.arrizabalaga@ehu.eus}

\author{L. Le Treust}
\address[L. Le Treust]{IRMAR, Universit\'e de Rennes 1, Campus de Beaulieu, F-35042 Rennes cedex, France}
\email{loic.letreust@univ-rennes1.fr}
\author{N. Raymond}
\address[N. Raymond]{IRMAR, Universit\'e de Rennes 1, Campus de Beaulieu, F-35042 Rennes cedex, France}
\email{nicolas.raymond@univ-rennes1.fr}

\keywords{Relativistic particle in a box, MIT bag model, Robin Laplacian}

\begin{document}

\maketitle

\begin{abstract}
This paper is devoted to the mathematical investigation of the MIT bag model, that is the Dirac operator on a smooth and bounded domain of $\R^3$ with certain boundary conditions. We prove that the operator is self-adjoint and, when the mass $m$ goes to $\pm\infty$, we provide spectral asymptotic results.
\end{abstract}
\tableofcontents

\section{Introduction}
\subsection{The physical context}
In elementary particles physics \cite{griffiths2008introduction}, the \emph{strong force} is one of the four known fundamental interaction forces along with the electromagnetism, the weak interaction and the gravitation. It is responsible for the confinement of the quarks inside composite particles called hadrons such as protons, neutrons or mesons. Its force-carrying (gauge bosons) particles are called the gluons (the force-carrying particles of the electromagnetism are the photons) and they carry together with the quarks, a type of charges called the color charges. Their interactions are detailed in the theory of \emph{quantum chromodynamics} (the theory of electromagnetism is called quantum electrodynamics). 

\subsubsection{The standard model}
Following the work of Gell-Mann and Zweig and the deep inelastic scattering experiments held at the Stanford Linear Accelerator Center in the $'60$, physicists introduced in the mid-$'70$, the standard model \cite{hosaka2001quarks} in an attempt to give a unified framework for the elementary particle physics. It turned out that this model has been very fruitful for it allowed to predict the existence of many particles. Despite of its success, the confinement of the quarks remains badly understood because of the complexity of the associated equations.
\subsubsection{An attempt to better understand the quarks confinement}
In parallel to the introduction of the standard model, Chodos, Jaffe, Johnson, Thorn, and Weisskopf \cite{MIT061974,MIT101974,MIT101975,johnson,hosaka2001quarks}, physicists at the MIT, developped a simplified phenomenological model to get a better understanding of the phenomenons involved in the quark-gluon confinement. Following the results of the experimentations held at that time, they chose to include several qualitative properties of the quarks: 
\begin{enumerate}[-]
	\item the perfect confinement of the quarks inside the hadrons\footnote{No isolated quark has been observed yet.},
	\item the relativistic nature of the quarks\footnote{For light quarks, the non-relativistic approximation $E = mc^2$ is not valid and the Schr\"odinger operator $-\Delta$ has to be replaced by the Dirac one to describe the kinetic energy.}.
\end{enumerate} 
The region of space $\Omega$ where the quarks live is called the bag and the model is the \emph{MIT bag model}. Let us remark that the MIT bag model can also be viewed as a model for a relativistic particle confined in a box. In the non-relativistic setting, the Dirac operator is replaced by the Dirichlet Laplacian and the associated model appears in many introduction courses of quantum physics \cite{MR0129790}. 
\\
Let us also mention that the two dimensional equivalent of the MIT bag model appears in the study of graphene and is referred to as the infinite mass boundary condition (see \cite{benguria2016spectral,stockmeyer2016infinite} and the references therein).
\subsection{The MIT bag Dirac operator}
In the whole paper,  $\Omega$ denotes a fixed bounded domain of $\R^3$ with regular boundary and $m$ is a real number. The Planck constant and the velocity of light are assumed to be equal to $1$.

Let us recall the definition of the Dirac operator associated with the energy of a relativistic particle of mass $m$ and spin $\frac{1}{2}$ (see \cite{Thaller1992}). The Dirac operator is a first order differential operator, acting on $L^2(\Omega,\mathbb{C}^4)$ in the sense of distributions, defined by 
\begin{equation}
H=\alpha\cdot D+m\beta \,,\qquad D=-i\nabla\,,
\end{equation}
where $\alpha=(\alpha_1,\alpha_2,\alpha_3)$, $\beta$ and $\gamma_5$ are the $4\times4$ Hermitian and unitary matrices given by
\[
	\beta=\left(\begin{array}{cc}1_2&0\\0&-1_2\end{array}\right),\;\gamma_5=\left(\begin{array}{cc}0&1_2\\1_2&0\end{array}\right),\;\alpha_k=\left(\begin{array}{cc}0&\sigma_k\\\sigma_k&0\end{array}\right) \mbox{ for }k=1,2,3\,.
\]
Here, the Pauli matrices $\sigma_1,\sigma_2$ and $\sigma_3$ are defined by
\[
	\sigma_1 = \left(
		\begin{array}{cc}
			0&1\\1&0
		\end{array}
	\right),
	\quad
	\sigma_2 = \left(
		\begin{array}{cc}
			0&-i\\i&0
		\end{array}
	\right),
	\quad
	\sigma_3 = \left(
		\begin{array}{cc}
			1&0\\0&-1
		\end{array}
	\right)\,,
\]
and $\alpha \cdot X$ denotes $\sum_{j=1}^3\alpha_j X_j$ for any $X = (X_1,X_2,X_3)$.
Let us now impose the boundary conditions under consideration in this paper and define the associated unbounded operator.
\begin{notation}
	In the following , $\Gamma := \pa \Omega$ and for all $\x\in\Gamma$, $\n(\x)$ is the outward-pointing unit normal to the boundary.
\end{notation}

\begin{definition}\label{def.MIT}
	The MIT bag Dirac operator $(H_m^\Omega,\mathcal{D}(H_m^\Omega))$ is defined on the domain
	\[
		\mathsf{Dom}(H_m^\Omega) = \{\psi\in H^1(\Omega,\mathbb{C}^4)~:~\mathcal{B}\psi = \psi~\text{on}~\Gamma\}\,,\qquad\mbox{ with }\quad\mathcal{B}=-i\beta(\alpha\cdot \n)\,,
	\]
	by $H_m^\Omega\psi=H\psi$ for all $\psi\in\D(H_m^\Omega)$. Note that the trace is well-defined by a classical trace theorem.
\end{definition}
\begin{notation} 
We will denote $H = H^\Omega_m$ when there is no risk of confusion. We denote $\braket{\cdot,\cdot}$ the $\C^4$ scalar product (antilinear w.r.t. the left argument) and $\braket{\cdot,\cdot}_{U}$ the $L^2$ scalar product on the set $U$.
\end{notation}
\begin{remark}\label{rk:boundfllux}
	The operator $\mathcal{B}$ defined for all $\x\in \Gamma$ is a Hermitian matrix which satisfies $\mathcal{B}^2 = 1_4$ so that its spectrum is $\{\pm 1\}$. Both eigenvalues have multiplicity two. Thus, the MIT bag boundary condition imposes the wavefunctions $\psi$ to be eigenvectors of $\mathcal{B}$ associated with the eigenvalues $+1$ . This boundary condition is chosen by the physicists \cite{johnson} so as to get a vanishing normal flow at the bag surface $-i\n\cdot \mathbf{j} = 0$ at the boundary $\Gamma$ where the current density $\mathbf{j}$ is defined by
	\[
		\mathbf{j} = \braket{\psi,\alpha \psi}.
	\]
	The opposite boundary condition $\psi \in \ker(1_4+\mathcal{B})$ is discussed in Section \ref{sec:negativemass}.
\end{remark}
Let us now describe our results.
\subsection{Results}

\subsubsection{Self-adjointness}
The following theorem gathers some fundamental spectral properties of the MIT bag Dirac operator that we establish in this paper and that are related to its self-adjointness.
\begin{theorem}\label{theo:selfadj}
	Let $\Omega$ be a nonempty, bounded and regular open set in $\R^3$ and $m\in \R$. The following properties hold true.
	\begin{enumerate}[\rm i.]
		\item \label{eq:theo11} The operator $(H,\D(H))$ is a self-adjoint operator with compact resolvent.
		\item \label{eq:theo12} There is a non-decreasing sequence $ (\mu_n(m))_{n\geq 1}\subset \R_+^*$ such that the spectrum of $H$, denoted by $\mathsf{sp}(H)$, is 
		\[
			\mathsf{sp}(H) = \{\pm\mu_{n}(m),\;n\geq 1\}\,.
		\]
		\item \label{eq:theo13} Each eigenvalue $\mu_{n}(m)$ has pair multiplicity.
		\item \label{eq:theo14} For each $\psi\in \D(H)$, we have
			\begin{equation}\label{eq.square1}
				\|H\psi\|^2_{L^2(\Omega)} = \|\alpha\cdot\nabla\psi\|^2_{L^2(\Omega)} + m\|\psi\|^2_{L^2(\pa\Omega)} + m^2\|\psi\|^2_{L^2(\Omega)}
			\end{equation}
			and
			\begin{equation}\label{eq.square2}
				\|\alpha\cdot\nabla\psi\|^2_{L^2(\Omega)} =  \|\nabla\psi\|^2_{L^2(\Omega)}+\frac{1}{2} \int_{\pa\Omega}\kappa|\psi|^2\dx s
			\end{equation}
			where $\kappa$ is the trace of the Weingarten map:
	\[
		\begin{array}{llll}
			d\n_s: &T_s \partial \Omega&\longrightarrow &\R^3 \\
			&v&\longmapsto&\partial_v \n(s)\,.
		\end{array}
	\]

	\end{enumerate}
\end{theorem}
Self-adjointness results have already been obtained in the case of $C^\infty$-boundaries in \cite{booss2009calderon} through Calder\'on projections and sophisticated pseudo-differential techniques, and, in two dimensions, with $C^2$ boundaries  \cite{benguria2016spectral} (see also \cite{stockmeyer2016infinite}), using Cauchy kernels and Riemann mapping theorem. 
The proof that we present here relies on simple PDE techniques and the introduction of an extension operator for Sobolev spaces (see \cite[Section 9.2]{brezis2010functional}) and can be generalized to any dimension. Let us also mention that more general local boundary conditions are considered in \cite{booss2009calderon,benguria2016spectral}.
\\
Our other results are of asymptotic nature. They describe the limiting behavior of the eigenvalues of the MIT bag Dirac operator as $m$ tends to $\pm\infty$.
\subsubsection{The MIT bag model with positive mass}

As we can guess from the expressions \eqref{eq.square1} and \eqref{eq.square2}, when $m\to+\infty$, the operator $H^2-m^2$ tends, in some sense, towards the Dirichlet Laplacian on $\Omega$. From the physical point of view, this limit is called the non-relativistic limit since it relates the MIT bag model (relativistic particles in a box) to the model for non-relativistic particles in a box.

From the spectral point of view, we have the following asymptotic result.
\begin{theorem}\label{theo.positive}
Let $-\Delta^\mathsf{Dir}$ be the Laplacian with domain $H^2(\Omega, \C)\cap H^1_{0}(\Omega, \C)$, and let $(\mu^\mathsf{Dir}_{n})_{n\geq 1}$ be the non-decreasing sequence of its eigenvalues. For all $n\geq 1$, we have
\[\mu_{n}(m)-\left(m+\frac{1}{2m}\mu^\mathsf{Dir}_{n}\right)\underset{m\to+\infty}{=}o\left(\frac{1}{m}\right)\,.\]
\end{theorem}
It is actually possible to describe the next term in the expansion of the first positive eigenvalue.
\begin{theorem}\label{theo.positive2}
	Let $u_1\in H^1_0(\Omega,\C)$ be a $L^2$-normalized eigenfunction of the Dirichlet Laplacian associated with its lowest eigenvalue $\mu_{1}^\mathsf{Dir}$. 
	We have
	\[
		\mu_{1}(m)-\left(m+\frac{1}{2m}\mu^\mathsf{Dir}_{1} -\frac{1}{2m^2}\int_{\Gamma}|\pa_\n u_1|^2\dx\Gamma \right)\underset{m\to+\infty}{=}o\left(\frac{1}{m^2}\right).
	\]
\end{theorem}
\begin{remark}
	This asymptotic expansion of $\mu_1(m)$ coincides with the one of the first eigenvalue of the operator $\sqrt{m^2 -\Delta^\mathsf{Rob}_{2m}}$ where $-\Delta^\mathsf{Rob}_{2m}$  is the Robin Laplacian of mass $2m$, \emph{i.e.} the operator of $L^2(\Omega,\C)$ whose quadratic form is defined for $u\in H^1(\Omega,\C)$ by 
	\[
		u\longmapsto\int_\Omega |\nabla u|^2\dx\x + 2m\int_\Gamma |u|^2 \dx\Gamma .
	\]
\end{remark}
\subsubsection{The MIT bag model with negative mass}\label{sec:negativemass}
Let us now describe our result related to the MIT bag model with \enquote{negative mass}. This \enquote{negative mass} may be understood in two equivalent ways.
\begin{enumerate}[\rm i.]
\item When we investigate the case $\Omega=\R^3$, the Dirac operators $\alpha\cdot D+m\beta$ and $\alpha\cdot D-m\beta$ are unitarily equivalent. Thus, in the case of a general $\Omega$, one may be tempted to consider $\alpha\cdot D-m\beta$ with the MIT bag condition $\mathcal{B}$.
\item Since we have 
\[
	\gamma_5\left(\alpha\cdot D + m\beta\right)\gamma_5 = \alpha\cdot D - m\beta\,,\qquad \gamma_5\mathcal{B}\gamma_5 = -\mathcal{B}\,,
\]
we notice that $\alpha\cdot D-m\beta$ with boundary condition $\mathcal{B}$ is unitarily equivalent to $\alpha\cdot D+m\beta$ with boundary condition $-\mathcal{B}$. In this case, the flux $-i\n\cdot \mathbf{j}$ also vanishes at the boundary and the justification given by the physicists \cite{johnson} of the MIT bag boundary condition can also be applied here (see Remark \ref{rk:boundfllux}).
\end{enumerate}
Of course, these changes of signs have no effect on the self-adjointness. 

\begin{remark}
	From the physical point of view, the fact that $H^\Omega_0$ does not commutes with the chirality matrix $\gamma_5$ (see \cite{Thaller1992}) is called the chiral symmetry violation \cite{Thomas1984,hosaka2001quarks}.  The chiral symmetry is supposed to be a property approximately satisfied by light quarks and exactly satisfied for quarks of mass $0$.
\end{remark}

In this paper, we will show that the limit $m\to-\infty$ for the operator $H_{m}^\Omega$ turns out to be a \emph{semiclassical limit} and not of perturbative nature as when $m\to+\infty$. It will be shown that the boundary is attractive for the eigenfunctions with eigenvalues lying essentially in the Dirac gap $[-|m|,|m|]$ and that their distribution is governed by the operator
\begin{equation}\label{eq.eff-}
\mathcal{L}^{\Gamma}-\frac{\kappa^2}{4}+K\,,
\end{equation}
where $\kappa$ and $K$ are the trace and the determinant of the Weingarten map, respectively, and where $\mathcal{L}^\Gamma$ is defined as follows.
\begin{definition}
The operator $(\mathcal{L}^\Gamma,\mathcal{D}(\mathcal{L}^\Gamma))$ is the operator associated with the quadratic form
\[\mathcal{Q}^\Gamma(\psi)=\int_{\Gamma}\|\nabla_{s}\psi\|^2\dx\Gamma\,,\quad\forall\psi\in H^1(\Gamma,\C)^4\cap\ker(\mathcal{B}-1_{4})\,.\]
\end{definition}
As a consequence of our investigation, we will get the following lower bound of the quadratic form $\mathcal{Q}^\Gamma$.
\begin{proposition}\label{prop.lbkK}
We have
\[\forall\psi\in H^1(\Gamma,\C)^4\cap\ker(\mathcal{B}-1_{4})\,,\quad \mathcal{Q}^\Gamma(\psi)\geq\int_{\Gamma}\left(\frac{\kappa^2}{4}-K\right)|\psi|^2\dx\Gamma\,.\]
\end{proposition}

Taking advantage of semiclassical technics, we will establish the following \emph{uniform} eigenvalues estimate.
\begin{theorem}\label{theo.bis}
Let $\eps_{0}\in(0,1)$ and
\[\mathsf{N}_{\eps_{0},m}:=\{n\in\N^* : \mu_{n}(-m)\leq m\sqrt{1-\eps_{0}}\}\,.\]
There exist positive constants $C_{-}$, $C_{+}$, $m_{0}$ such that, for all $m\geq m_{0}$ and $n\in\mathsf{N}_{\eps_{0},m}$,
\[\mu^-_{n}(m)\leq\mu_{n}(-m)\leq\mu^+_{n}(m)\,,\]
with $\mu^\pm_{n}(m)$being the $n$-th eigenvalue of the operators  $\mathcal{L}_{m}^{\Gamma,\pm}$ of $L^2(\Gamma,\C)^4$ defined by
\begin{align*}
\mathcal{L}_{m}^{\Gamma,-}=&\left([1-C_{-}m^{-\frac{1}{2}}]\mathcal{L}^{\Gamma} 
-\frac{\kappa^2}{4}+K-C_{-}m^{-1}\right)_{+}^{\frac{1}{2}}\,,\\
\mathcal{L}_{m}^{\Gamma,+}=&\left([1+C_{+}m^{-\frac{1}{2}}]
\mathcal{L}^{\Gamma} 
-\frac{\kappa^2}{4}+K+C_{+}m^{-1}\right)^{\frac{1}{2}}\,.
\end{align*}
\end{theorem}
\begin{remark}
By Proposition \ref{prop.lbkK}, 
\[[1+C_{+}m^{-\frac{1}{2}}]\mathcal{L}^{\Gamma} -\frac{\kappa^2}{4}+K+C_{+}m^{-1}\geq 0\,,\]
so that the square root is well-defined. In the expression of $\mathcal{L}_{m}^{\Gamma,-}$, we are obliged to take the non-negative part.
\end{remark}
Rewriting the previous theorem in term of asymptotic expansions of the eigenvalues, we get the following result (see for instance \cite[Corollary 3.2]{KKR16}).
\begin{corollary}
	For all $n\in \N^*$, we have that
	\[
		\mu_n(-m) \underset{m\to+\infty}{=} \widetilde{\mu}^{\frac{1}{2}}_n +\mathcal{O}(m^{-\frac{1}{2}}),
	\]
	where $(\widetilde{\mu}_n)_{n\in \N^*}$ is the non-decreasing sequence of the eigenvalues of the following non-negative operator on $L^2(\Gamma,\C)^4\cap \ker(1_4-\mathcal{B})$: 
	\[\mathcal{L}^{\Gamma}-\frac{\kappa^2}{4}+K\,.\]
\end{corollary}
Let us describe the spectrum of the effective operator on the boundary in the case where $\Omega$ is a ball (see \cite[Section 4.6]{Thaller1992}).The proof of the following proposition just follows from straightforward computations.
\begin{proposition}
	Assume that $\Omega = B(0,R)$ with $R>0$. Let $A = \beta(1+ 2S\cdot L)$ be the "spin-orbit" operator where $S = \frac{1}{2}\gamma_5 (\alpha_1,\alpha_2,\alpha_3)$ and $L = \x\times D$. We have 
	\[
		A\mathcal{B} = \mathcal{B}A\,,
	\]
	\[
		\mathcal{L}^{\Gamma}-\frac{\kappa^2}{4}+K = R^{-2}A^2\,,
	\]
	and its spectrum is $\{n^2/R^2,\,n\in \N^*\}$.
\end{proposition}
\subsection{Remarks}
Let us conclude this introduction with some comments related to Robin Laplacians, $\delta$-interactions and shell-interactions.
\subsubsection{Comparison to Robin Laplacians}
Theorem \ref{theo.bis} shares common features with the known results about the Robin Laplacian in the strong coupling limit (see \cite{PP16} for the asymptotic of individual eigenvalues and \cite{KKR16} in relation with the spectral uniformity and the semiclassical point of view). But two major differences have to be emphasized. Firstly, the effective operator is \emph{not semiclassical} in our case (it looks like the effective operator in the case of a Schr\"odinger operator with a strong attractive $\delta$-interaction on $\Gamma$, see \cite{DEKP16}). Secondly, the effective operator in our case is a quadratic function of the principal curvatures (and not a linear one as in the Robin case). These differences are crucially related to the vectorial nature of the Dirac operator with the MIT conditions:  they lead to a kind of semiclassical degeneracy. It is also rather surprising that the order of this degeneracy is still less than the order of the famous Born-Oppenheimer correction. Here, by the Born-Oppenheimer method, we mean a semiclassical method of reduction to the boundary explained in Sections \ref{sec.4} and \ref{sec.5}.

\subsubsection{Shell interactions}\label{rel. shell int}
There is a close relation between the MIT bag model that we study in this work and the shell interactions for Dirac operators studied in \cite{AMV15}. In \cite[Theorem 5.5]{AMV15}, the authors prove that $H+V_{es}$ generates confinement with respect to $\Gamma$ for $\lambda_e^2-\lambda_s^2=-4$, where 
\[\displaystyle{V_{es}\psi=\frac{1}{2}\,(\lambda_e+\lambda_s\beta)(\psi_++\psi_-)}\dx\Gamma\,,\]
$\lambda_e,\lambda_s\in \mathbb{R}$, $\psi_{\pm}$ are the non-tangential boundary values of $\psi$ on $\Gamma$ and $\dx\Gamma$ is the surface measure on $\Gamma$. By using \cite[Proposition 3.1]{AMV15}, it is possible to see that the existence of eigenvalues for $H+V_{es}$ is equivalent to a spectral property of some bounded operators on $\Gamma$. More precisely, 
\begin{equation}\label{remark1 eq1}
\ker(H+V_{es}-\mu)\neq 0 \Longleftrightarrow \ker(\lambda_s\beta-\lambda_e+4C_{\sigma,\mu})\neq 0\,,
\end{equation}
where $C_{\sigma,\mu}$ is a Cauchy-type operator defined on $\Gamma$ in the principal value sense.
In the regime $\lambda_e^2-\lambda_s^2=-4$, the right hand side of $(\ref{remark1 eq1})$ is also equivalent to the existence of a solution $\psi\in H^1(\Omega, \mathbb{C}^4)$ of the boundary value problem $(H-\mu)\psi=0$ in $\Omega$ and $\psi=\frac{i}{2}\,(\lambda_e-\lambda_s\beta)(\alpha\cdot \n)\psi$ on $\Gamma$. Observe that when $\lambda_e=0$ and $\lambda_s=2$ we recover the MIT bag model given in Definition $\ref{def.MIT}$. It is worth pointing out that the right hand side of $(\ref{remark1 eq1})$ does not hold for $\lambda_s>0$ if $\mu\in[-m,m]$. So the eigenvalues must belong to $\mathbb{R}\backslash [-m, m]$ for $\lambda_s>0$, as we already know from \cite[Section 5]{LTPhd} in the case $\lambda_e=0$ and $\lambda_s=2$.

\subsection{Organization of the paper}
The paper is organized as follows. In Section \ref{sec.2}, we prove Theorem \ref{theo:selfadj} by constructing extension operators adapted to the Dirac operator. Section \ref{sec.3} is devoted to the proofs of Theorems \ref{theo.positive} and \ref{theo.positive2}. The remaining sections are concerned with the case of the large negative mass. In Section \ref{sec.4}, we explain the main steps towards the proof of Theorem \ref{theo.bis}. In Section \ref{sec.5}, we prove the propositions and theorems stated in Section \ref{sec.4}.

\section{Proof of Theorem \ref{theo:selfadj}}\label{sec.2}
\subsection{Preliminaries}
This section is devoted to establish elementary algebraic properties.
\begin{lemma}\label{eq:multalp}
For all  $\x,\y\in\R^3$, we have
\begin{equation*}
	\begin{split}
	&
	(\alpha\cdot \x)(\alpha \cdot \y) = (\x\cdot \y)1_4 + i\gamma_5\alpha\cdot (\x\times \y)\,,
	\\
	&
	\beta (\alpha\cdot \x) = - (\alpha\cdot \x) \beta\,,\quad \beta \gamma_5 = -\gamma_5\beta\,,\\
	&
	\gamma_5(\alpha\cdot \x) = (\alpha\cdot \x)\gamma_5\,.
	\end{split}
\end{equation*}
\end{lemma}
\begin{proof}
We refer to \cite[Appendix 1.B]{Thaller1992}.
\end{proof}
Points \ref{eq:theo12} and \ref{eq:theo13} of Theorem \ref{theo:selfadj} are immediate consequences of the following lemma (see \cite[Section 1.4.6]{Thaller1992} and \cite[Section 10.4.5]{reiher2014relativistic}).
\begin{lemma}[Discrete symmetries] Let us introduce three operators defined for $\psi\in \C^4$ by 
	\[
	\begin{split}
		&C\psi = i\beta\alpha_2\overline{\psi}, \quad\mbox{Charge conjugation operator,}\\
		&T\psi  = -i\gamma_5\alpha_2\overline{\psi}, \quad \mbox{Time reversal-symmetry operator,}\\
		&CT\psi  = \beta \gamma_5 \psi, \quad \mbox{CT-symmetry operator.}
	\end{split}
	\] 
	The operators $C$ and $T$, \it{resp.}~$CT$ are anti-unitary, resp.~unitary transformations that leave $\D(H)$  invariant and satisfy $C^2 = -T^2 = 1_4$, $CT=TC$,
	\[
		HC = -CH,\quad HT = TH\quad \mbox{ and }\quad H\left(CT\right) = -\left(CT\right)H\,.
	\]
	Moreover, we have for any $\psi\in \C^4$ that $\braket{\psi, T\psi} = 0$.
\end{lemma}

We can relate the mean curvature to the commutator between the boundary condition and a Dirac derivative parallel to the boundary.
\begin{lemma}[Mean curvature as commutator]\label{eq.anticurv}
We have
\[[\alpha\cdot(\n\times D),\mathcal{B}]=-\kappa\gamma_{5}\mathcal{B}\,.\]
\end{lemma}
\begin{proof}
Let $s\in \partial \Omega$.
First we have, by anticommutation between $\alpha$ and $\beta$,
\[\alpha\cdot(\n\times D)\mathcal{B}\psi=\beta\,\alpha\cdot(\n\times\nabla)(\alpha\cdot\n\,\psi)\,.\]
Let $\n'$ and $\n''$ be two eigenvectors of the Weingarten map $d\n_s$ whose respective eigenvalues are denoted by $(\lambda',\lambda'')$ and such that $(\n,\n',\n'')$ is an orthonormal basis of $\R^3$.
We have
\[\alpha\cdot(\n\times\nabla)=\alpha\cdot \n''\partial_{\n'}-\alpha\cdot\n'\partial_{\n''}\,.\]
Then, by the Leibniz formula and Lemma \ref{eq:multalp}, it follows that
\[
\begin{split}
	(\alpha\cdot\n\times\nabla)(\alpha\cdot\n\,\psi)
	&=-\alpha\cdot\n\left(\alpha\cdot \n''\partial_{\n'}-\alpha\cdot\n'\partial_{\n''}\right)\psi
	\\
	&\qquad \qquad+\left((\alpha\cdot\n'')(\alpha\cdot\partial_{\n'}\n)-(\alpha\cdot\n')(\alpha\cdot\partial_{\n''}\n)\right)\psi\,,
\end{split}
\]
and thus, again by Lemma \ref{eq:multalp},
\[(\alpha\cdot\n\times\nabla)(\alpha\cdot\n\,\psi)=-\alpha\cdot\n\left(\alpha\cdot \n''\partial_{\n'}-\alpha\cdot\n'\partial_{\n''}\right)\psi-i(\lambda'+\lambda'')\gamma_{5}\alpha\cdot\n\,.\]
We deduce that
\[\alpha\cdot(\n\times D)\mathcal{B}\psi=\mathcal{B}(\alpha\cdot\n\times D)\psi-i(\lambda'+\lambda'')\beta\gamma_{5}\alpha\cdot\n\,,\]
and the conclusion follows.
\end{proof}

	\subsection{Symmetry of $H$}
	Let us start by proving the symmetry of $H$.

	\begin{lemma}\label{lem:sym}
	$(H,\D(H))$ is a symmetric operator.
\end{lemma}
\begin{proof}
	Since the $\alpha$-matrices are Hermitian, we have, thanks to the Green-Riemann formula:
	\begin{equation}\label{eq.IBP0}
	\forall\varphi,\psi\in H^1(\Omega,\C^4)\,,\qquad\braket{\alpha\cdot D\varphi,\psi}_\Omega=\braket{\varphi,\alpha\cdot D\psi}_\Omega+\braket{(-i\alpha\cdot\n)\varphi,\psi}_{\partial\Omega}\,.
	\end{equation}
	Now we consider $\psi,\varphi\in\D(H)$. By using $\beta^2=1_{4}$ and the boundary condition, we get
	\[\braket{(-i\alpha\cdot\n)\varphi,\psi}_{\partial\Omega}=\braket{\beta\varphi,\psi}_{\partial\Omega}\,,\]
	so that, we deduce
	\begin{equation}\label{eq.IBP1}
	\forall \varphi,\psi\in\mathcal{D}(H)\,,\qquad\braket{\alpha\cdot D\varphi,\psi}_\Omega-\braket{\varphi,\alpha\cdot D\psi}_\Omega=\braket{\beta\varphi,\psi}_{\partial\Omega}\,.
	\end{equation}
	The right hand side of \eqref{eq.IBP1} is a skew-symmetric expression of $(\varphi, \psi)$ and the left hand side is symmetric in $(\varphi, \psi)$ since $\beta$ is Hermitian.
	Thus both sides must be zero.
\end{proof}

	\subsection{Self-adjointness of $H$}
	This subsection is devoted to the proof of Point \ref{eq:theo11} of Theorem \ref{theo:selfadj}. 
	We denote by $\mathscr{L}(E,F)$ the set of continuous linear applications from $E$ to $F$ where $E$ and $F$ are Banach spaces. We recall that the domain of $H$ is independent of $m$: 
	\[
		\D(H) = \{\psi\in H^1(\Omega)^4,\;\mathcal{B}\psi = \psi \mbox{ on }\pa\Omega\}\,,
	\]
	and that the domain of the adjoint $H^\star$ is defined by
	\[
		\D(H^\star) = \{\psi\in L^2(\Omega)^4,\; L_\psi\in \mathscr{L}(L^2(\Omega)^4,\C)\}\,,
	\]
	where 
	\[
		L_\psi : \varphi\in\D(H)\mapsto \braket{\psi,H\varphi}_\Omega\in \C\,.
	\]
	By Lemma \ref{lem:sym}, we get that 
	\[
		\D(H)\subset\D(H^\star)\,.
	\]
	Let us remark that, without loss of generality, we can assume in the proof that $m = 0$ since the operator $\beta m $ is bounded (and self-adjoint) from $L^2(\Omega)^4$ into itself.
         The aim of this section is to establish that
	\begin{equation}\label{eq.HsH}
		\D(H^\star)\subset\D(H)\,.
	\end{equation}

\subsubsection{Extension operator on the half-space case}\label{sec:halfspace}
In this section, we consider the case when $\Omega=\R^3_{+}$ and we establish the existence of an extension operator.
\begin{lemma}\label{lem.extHS}
There exists an operator
                           \[
				P : \D(H^\star)\to \{ \psi\in L^2(\R^3)^4,\; \alpha\cdot D \psi\in L^2(\R^3)^4\} = H^1(\R^3)^4
			\]
			such that $P\psi_{|\R^3_+} = \psi$ and 
			\[
				\|P\psi\|^2_{H^1(\R^3)}= \|P\psi\|^2_{L^2(\R^3)} + \|\nabla P\psi\|^2_{L^2(\R^3)} 
				= 2\left(\|\psi\|^2_{L^2(\R^3_+)} + \|\alpha\cdot D \psi\|^2_{L^2(\R^3_+)}\right).
			\]
\end{lemma}

\begin{proof}
	The outward-pointing normal $\n$ is equal to $-e_3 = (0,0,-1)^T$ so that the boundary condition is
	\[
		i\beta\alpha_3\psi = \psi\,,
	\]
	on $\pa \R ^3_+$.
	Let us diagonalize the matrix $i\beta\alpha_3$ appearing in the boundary condition. We introduce the matrix
	\[
		T = \frac{1}{\sqrt{2}}\left(
			\begin{array}{cc}
				1_2 & i1_2\\
				i1_2& 1_2
			\end{array}
		\right)\,.
	\]
	We have
	\[
		T\beta T^\star = \left(
			\begin{array}{cc}
				0 & -i1_2\\
				i1_2& 0
			\end{array}
		\right)
		\,,\qquad
		T\alpha_kT^\star = \alpha_k\,,\quad T(i\beta\alpha_{3})T^\star=\left(
			\begin{array}{cc}
				\sigma_3 & 0\\
				0& -\sigma_3
			\end{array}
		\right) =: \mathcal{B}^0\,.
	\]
	Thus we consider $\widetilde H=THT^\star$. The operator $\widetilde H$ is defined by $\widetilde H \psi = \alpha\cdot D \psi $ for any $\psi\in \D(\widetilde H)$ where
	\begin{equation}\label{eq:domsuper}\begin{split}
		\D(\widetilde H) &= \left\{\psi\in H^1(\R^3_+),\;
		\mathcal{B}^0
		\psi = \psi,\mbox{ on }\pa\R^3_+
\right\}
		\\
		&= \left\{\psi\in H^1(\R^3_+),\;\psi^2 = \psi^3 = 0\mbox{ on }\pa\R^3_+
\right\}
	\end{split}\end{equation}
	and $\psi = (\psi^1,\psi^2,\psi^3,\psi^4)^T$. This unitarily equivalent representation of the Dirac operator is called the \emph{supersymmetric representation} (see \cite[Appendix 1.A]{Thaller1992}). This expression of the domain makes more apparent the fact that the MIT bag boundary condition is intermediary between the Dirichlet and Neumann boundary conditions.

	Let us denote by $S:\R^3\to \R^3$ and $\Pi : \R^3 \to \R^3$ the orthogonal  symmetry with respect to $\pa\R ^3_+$ and the orthogonal projection on $\pa \R ^3_+$.	Based on \eqref{eq:domsuper}, we define the extension operator $\widetilde P$ for $\psi\in\D(\widetilde H^\star)$ as follows:
	\[
		\widetilde P \psi (x,y,z) = \left\{
			\begin{array}{ll}
				\psi(x,y,z),\; &\mbox{if }z> 0
				\\
				(\psi^1,-\psi^2,-\psi^3,\psi^4)^T(x,y,-z) = \mathcal{B}^0 \left( {\psi}\circ S\right) (x,y,z),\;&\mbox{if }z< 0
			\end{array}
		\right.
	\]
	for $(x,y,z)\in \R^3$. In other words, we extend $\psi^1, \psi^4$ by symmetry and $\psi^2, \psi^3$ by antisymmetry.

	Let us get back to the standard representation and define the extention operator $P$ for $\psi\in \mathcal{D}(H^\star)$ and $(x,y,z)\in \R^3$ as follows : 
	\[
		P\psi(x,y,z) = T^\star\widetilde {P}T\psi(x,y,z) = 
		\begin{cases}
			\psi(x,y,z),&\mbox{ if }z>0,\\
			\left( \mathcal{B}\circ \Pi\right)  \left( {\psi}\circ S\right) (x,y,z), &\mbox{ if }z<0.
		\end{cases}
	\]
	Since $\mathcal{B}(s)$ is a unitary transformation of $\C^4$ for any $s\in \pa\R_+^3$, we get that
	\[
		\|P\psi\|_{L^2(\R^3)}^2 = 2\|\psi\|^2_{L^2(\R^3_+)}.
	\]
	Let us study $\alpha\cdot D P\psi$ in the distributional sense. We have for $\varphi\in \mathcal{D} = C^\infty_0(\R^3)$ that
	\[\begin{split}
		\braket{\alpha\cdot D P\psi,\varphi}_{\mathcal{D}'\times \mathcal{D}}
		 = \braket{ P\psi,  \alpha\cdot D\varphi}_{\R^3}
		 = \braket{\psi,\alpha\cdot D \varphi}_{\R^3_+} + \braket{	\left( \mathcal{B}\circ \Pi\right) {\psi}\circ S,\alpha\cdot D \varphi}_{\R^3_-}
	\end{split}\]
	where $\braket{\cdot,\cdot}_{\mathcal{D}'\times \mathcal{D}}$ is the distributional bracket on $\R^3$.
	Since $\mathcal{B}$ is Hermitian, we obtain by a change of variables, that
	\[\begin{split}
		&\braket{	\left( \mathcal{B}\circ \Pi\right) {\psi}\circ S,\alpha\cdot D \varphi}_{\R^3_-}
		 = \braket{{\psi}\circ S,	\left( \mathcal{B}\circ \Pi\right) \alpha\cdot D \varphi}_{\R^3_-}
		 \\
		 &
		 \qquad= \braket{{\psi},	-i\left( \mathcal{B}\circ \Pi\right) \left(\alpha_1\pa_x + \alpha_2\pa_y-\alpha_3\pa_z\right){\varphi}\circ S}_{\R^3_+}
		 = \braket{{\psi}, \alpha\cdot D \left( \left( \mathcal{B}\circ \Pi\right) {\varphi}\circ S\right) }_{\R^3_+}.
	\end{split}\]
	Hence, we get
	\[\begin{split}
		\braket{\alpha\cdot D P\psi,\varphi}_{\mathcal{D}'\times \mathcal{D}} = \braket{{\psi}, \alpha\cdot D\left(\varphi + \left( \mathcal{B}\circ \Pi\right){\varphi}\circ S\right)}_{\R^3_+}.
	\end{split}\]
	Let us remark that the function $\varphi + \left( \mathcal{B}\circ \Pi\right){\varphi}\circ S$ belongs to $\D( H)$. Indeed, we have that
	\[
		\left( \mathcal{B}\circ \Pi\right)\left(\varphi + \left( \mathcal{B}\circ \Pi\right){\varphi}\circ S\right)(x,y,0) = \left(\varphi + \left( \mathcal{B}\circ \Pi\right){\varphi}\circ S\right)(x,y,0)
	\]
	for all $(x,y)\in \R^2$. Since $\psi\in \D( H^\star)$, by the Riesz theorem and a change of variable, we have that
	\[\begin{split}
		\braket{\alpha\cdot D P\psi,\varphi}_{\mathcal{D}'\times \mathcal{D}} 
		&= \braket{\alpha\cdot D{\psi},\left(\varphi + 	\left( \mathcal{B}\circ \Pi\right){\varphi\circ S}\right)}_{\R^3_+} 
		\\
		&=  \braket{\alpha\cdot D{\psi},\varphi}_{\R^3_+} +  \braket{	\left( \mathcal{B}\circ \Pi\right){\left(\alpha\cdot D{\psi}\right)\circ S},\varphi}_{\R^3_-}.
	\end{split}\]
	Thus, we obtain that in the distributional sense
	\[
		\alpha\cdot D P\psi = \chi_{\R^3_+}\left(\alpha\cdot D{\psi}\right) + 	\chi_{\R^3_-}\left( \mathcal{B}\circ \Pi\right){\left(\alpha\cdot D{\psi}\right)\circ S}\in L^2(\R^3)
	\]
	so that
	\[
		\|\nabla P\psi\|^2_{L^2(\R^3)} = \|\alpha\cdot D P\psi\|^2_{L^2(\R^3)} = 2\|\alpha\cdot D \psi\|^2_{L^2(\R^3_+)}.
	\]
\end{proof}
\subsubsection{Proof of Point \ref{eq:theo11} of Theorem \ref{theo:selfadj}}
Let us now consider the case of our general $\Omega$. Let us remark that the understanding of the case of the half-space is not sufficient to conclude since curvature effects have to be taken into account. The proof of Lemma \ref{lem.extHS} is just used here as a guideline for the proof of the next proposition.
\begin{proposition}\label{prop.extension}
There exist a constant $C>0$ and an operator
                           \[
				P : \D(H^\star)\to H^1(\R^3)^4
			\]
			such that $P\psi_{|\Omega} = \psi$ and 
			\[
				\|P\psi\|^2_{H^1(\R^3)}\leq C\left(\|\psi\|^2_{L^2(\Omega)} + \|\alpha\cdot D \psi\|^2_{L^2(\Omega)}\right)\,,
			\]
			for all $\psi\in\D(H^\star)$.
\end{proposition}
\begin{proof}
Using a partition of unity and the fact that
\[
	\{u\in L^2(\R^3)^4:\; \alpha\cdot D u\in L^2(\R^3)^4\} =  H^1(\R^3)^4,
\]
we are reduced to study the case of a deformed half-space. 
Let us recall the standard tubular coordinates near the boundary of $\Omega$ : 
\[
	\begin{split}
		\eta : &\left( U\cap\partial \Omega \right)\times (-T,T) \longrightarrow U,\\
		&(\x_0,t) \mapsto \x_0 - t\n (\x_0)
	\end{split}
\]
where $T>0$ and $U$ is a bounded open set of $\R^3$. Without loss of generality, we can assume that $\eta$ is a diffeomorphism such that
\[
	\eta(\left( U\cap\partial \Omega \right)\times (0,T)) = \Omega\cap U, 
	\quad 
	\eta (\left( U\cap\partial \Omega \right) \times \{0\}) = \partial \Omega \cap U.
\]
The rest of the proof is divided into four steps: 
\begin{enumerate}[(a)]
	\item \label{stepa}\emph{we introduce a bounded extension operator $P : L^2(U\cap\Omega)\to L^2(U)$,
	}
	\item \label{stepb}\emph{we introduce a map $\tilde \alpha$ which extends the $\alpha$-matrices on $U$ so that, we have 
	\[
		\|\tilde \alpha\cdot DP\psi\|_{L^2(U)}\leq C\left(\|\psi\|^2_{L^2(\Omega\cap U)} + \|\alpha\cdot D \psi\|^2_{L^2(\Omega\cap U)}\right),
	\]
	 for any function $\psi\in \D(H^\star)$ whose support is a compact subset of $U\cap \overline{\Omega}$,
	}
	\item \label{stepc}\emph{we show that the norm $\|\cdot\|_{\mathcal{V}}$ defined on
	\[
	 	\mathcal{V} = \{v\in L^2(U), \; \tilde \alpha\cdot D v\in L^2(U), \, \mathsf{supp} \,v\subset\subset U\}
	\]
	by
	\[
		\|v\|_{\mathcal{V}}^2 = \|v\|_{L^2}^2 + \|\tilde \alpha\cdot D v\|_{L^2}^2
	\]
	is equivalent to the $H^1$ norm on $C^\infty_0(U)$.
	}
	\item \label{stepd}\emph{we deduce by a density argument that $\mathcal{V} \subset H^1_0(U)$},
\end{enumerate}
and the conclusion follows.
\\
\emph{Step \eqref{stepa}.} The following tubular projection and symmetry defined by
\[
\begin{split}
	\phi_s :&\;U \to \Omega\cap U\\
	&\x_0-t\n (\x_0)\mapsto \x_0+t\n (\x_0)\\
	\phi_p : &\;U \to \pa\Omega\cap U\\
	&\x_0-t\n (\x_0)\mapsto \x_0
\end{split}
\]
 are well-defined and regular functions.

Let us denote by $P(\x_0)$ the matrix of the identity map of $\R^3$ from the canonical basis $(e_1,e_2,e_3)$ to the orthonormal basis 
$(\epsilon_1(\x_0),\epsilon_2(\x_0),\n(\x_0))$ defined by
\[
	P(\x_0) = \Mat(\mathsf{Id}, (e_1,e_2,e_3), (\epsilon_1(\x_0),\epsilon_2(\x_0),\n(x_0)))\,,
\]
for any $\x_0\in \pa\Omega \cap U$ where $(\epsilon_1(\x_0),\epsilon_2(\x_0))$ is a basis of the tangent space $T_{\x_0}\pa\Omega$.
Up to taking a smaller $T$, we have that, for any $\x_0\in \pa \Omega \cap U$,
\[
	\jac\;\phi_s(\x_0) = P(\x_0)^{-1}
	\left(
		\begin{array}{ccc}
			1&0&0\\
			0&1&0\\
			0&0&-1
		\end{array}
	\right)
	P(\x_0)\,,
\]
and, for any $\x\in U$,
\begin{equation}\label{eq:detjac}
	\frac{3}{2}\geq |\jac\;\phi_s(\x)| := |\det \jac\;\phi_s(\x)|\geq \frac{1}{2}\,.
\end{equation}
Following the idea of the proof of Lemma \ref{lem.extHS}, we define the extension operator 
\[
	P : L^2(U\cap \Omega)\to L^2(U)
\]
 for $\psi\in L^2(U\cap \Omega)$ and $\x\in U$ as follows: 
\[
	P\psi(\x) = 
	\begin{cases}
		\psi(\x), &\mbox{ if }\x\in U\cap \Omega\,,\\
		\left( \mathcal{B}\circ \phi_p(\x)\right) \psi\circ \phi_s (\x), & \mbox{ if }\x \in U\cap \Omega^c .
	\end{cases}
\]
By \eqref{eq:detjac}, we get that
\[
	\|P\psi\|_{L^2(U)}\leq C \|\psi\|_{L^2(U\cap \Omega)}\,.
\]
\emph{Step \eqref{stepb}.} Let us \emph{extend} the $\alpha$-matrices for $\x\in U$ as follows: 
\[
\widetilde \alpha (\x) = 
\begin{cases}
	(\alpha_1,\alpha_2,\alpha_3)^T\,, &\mbox{ if }\x\in U\cap \Omega,\\
	|\jac \,\phi_s(\x)|\mathcal{B}\circ \phi_p(\x)\left( \left( \jac \,\phi_s(\x)\right) ^{-1}(\alpha_1,\alpha_2,\alpha_3)^T\right) \mathcal{B}\circ \phi_p(\x)\,, &\mbox{ if }\x \in U\cap \Omega^c\,.
\end{cases}
\]
Let us remark that $\widetilde \alpha(\x)$ is a column-vector of three matrices and the above matrix product makes sense as a product in the modulus on the ring of the $4\times 4$ Hermitian matrices.
In particular, we get for $\x_0\in \pa\Omega\cap U$ that
\[
	\begin{split}
		&
		|\jac \,\phi_s(\x_0)|\mathcal{B}\circ \phi_p(\x_0)\left( \left( \jac \,\phi_s(\x_0)\right) ^{-1}(\alpha_1,\alpha_2,\alpha_3)^T\right) \mathcal{B}\circ \phi_p(\x_0)
		\\
		&\quad
		=
		\mathcal{B}(\x_0)\left( 
		P(\x_0)^{-1}
		\left(
		\begin{array}{ccc}
			1&0&0\\
			0&1&0\\
			0&0&-1
		\end{array}
		\right)
		P(\x_0)
		\left(
			\begin{array}{c}
				\alpha_1\\
				\alpha_2\\
				\alpha_3
			\end{array}
		\right)
		\right) \mathcal{B}(\x_0)
		\\
		&\quad
		=
		\mathcal{B}(\x_0)\left( 
		P(\x_0)^{-1}
		\left(
		\begin{array}{ccc}
			1&0&0\\
			0&1&0\\
			0&0&-1
		\end{array}
		\right)
		\left(
			\begin{array}{c}
				\alpha\cdot \epsilon_1(\x_0)\\
				\alpha\cdot \epsilon_2(\x_0)\\
				\alpha\cdot \n(\x_0)
			\end{array}
		\right)
		\right) \mathcal{B}(\x_0)
		\\
		&\quad
		=
		P(\x_0)^{-1}
		\mathcal{B}(\x_0)
		\left(
			\begin{array}{c}
				\alpha\cdot \epsilon_1(\x_0)\\
				\alpha\cdot \epsilon_2(\x_0)\\
				-\alpha\cdot \n(\x_0)
			\end{array}
		\right)
		\mathcal{B}(\x_0)
		\\
		&\quad
		=
		P(\x_0)^{-1}
		\left(
			\begin{array}{c}
				\alpha\cdot \epsilon_1(\x_0)\\
				\alpha\cdot \epsilon_2(\x_0)\\
				\alpha\cdot \n(\x_0)
			\end{array}
		\right)
		=
		\left(
			\begin{array}{c}
				\alpha_1\\
				\alpha_2\\
				\alpha_3
			\end{array}
		\right).
	\end{split}
\]
Hence, the application $\tilde \alpha$ is continuous on $U$. Since it is also a $C^1$-map on both $\overline{\Omega\cap U}$ and $\overline{\Omega^c\cap U}$, we get that $\tilde \alpha$ is a Lipschitz map.
This choice for the extension of $\alpha$ is made in order to get 
\[\widetilde{\alpha}\cdot DP\psi \in L^2(U)\,,\]
 in the sense of distributions. 
Indeed, since $\tilde \alpha$ is Lipschitz, we get that, for $\varphi\in H^1_0(U)$,
 \[
 \begin{split}
 &
 \braket{\widetilde{\alpha}\cdot D P\psi,\varphi}_{H^{-1}(U)\times H^{1}_0(U)}
 :=
	\braket{P\psi,\widetilde{\alpha}\cdot D \varphi}_{U}
 +\braket{P\psi,-i\dive(\widetilde{\alpha}) \varphi}_{U\cap\Omega^c}
 \end{split}\,.
 \] 
For $\x\in U\cap \Omega$, we also have that
\[
(\widetilde \alpha \cdot \nabla \varphi)(\phi_s^{-1}(\x))=|\jac \,\phi_s(\phi_s^{-1}(\x))|\,\left( \mathcal{B}\circ \phi_p\alpha\mathcal{B}\circ \phi_p\right) \cdot \nabla \left( \varphi \circ \phi_s^{-1}\right) (\x)
\]
and thus 
\begin{multline*}
(\widetilde \alpha \cdot \nabla \varphi)(\phi_s^{-1}(\x))=|\jac \,\phi_s(\phi_s^{-1}(\x))|\,\mathcal{B}\circ \phi_p\left( \alpha\cdot \nabla \left((\mathcal{B}\circ \phi_p)  \varphi \circ \phi_s^{-1}\right) \right)  (\x)\\ 
-  |\jac \,\phi_s(\phi_s^{-1}(\x))|\,\mathcal{B}\circ \phi_p\left( \alpha\cdot \nabla (\mathcal{B}\circ \phi_p) \right) \varphi \circ \phi_s^{-1} (\x)\,.
\end{multline*}
We deduce that
\begin{multline*}
	\braket{P\psi, \widetilde \alpha \cdot D \varphi }_{U\cap \Omega^c}= \braket{\psi, \alpha\cdot D \left( \left( \mathcal{B}\circ \phi_p\right)\varphi\circ \phi_s^{-1}\right)}_{U\cap \Omega}\\
 -\braket{\psi, \left(\alpha\cdot D \left( \mathcal{B}\circ \phi_p\right) \right)\varphi\circ \phi_s^{-1}}_{U\cap \Omega}\,.
\end{multline*}
Since $\psi\in \D(H^\star)$ and the function
$
	\varphi +  \left( \mathcal{B}\circ \phi_p\right)\varphi\circ \phi_s^{-1} : \Omega\cap U\to \C^4
$
 belongs to $\D(H)$, we get that
\[
\begin{split}
	\braket{\widetilde{\alpha}\cdot D P\psi, \varphi}_{H^{-1}(U)\times H^{1}_0(U)}
	=
	\braket{\alpha\cdot D \psi, \varphi +  \left( \mathcal{B}\circ \phi_p\right)\varphi\circ \phi_s^{-1}} _{U\cap\Omega}
	+\braket{\psi, R \varphi}_{U\cap\Omega}\,,
\end{split}
\]
where $R$  is a bounded operator from $L^2(U)$ in $L^2(U\cap \Omega)$ defined for all $\varphi\in L^2(U)$ by
\[
	R\varphi = -i\,\dive (\widetilde \alpha) \varphi+  i\left( 
	\alpha \cdot \nabla\left(  \mathcal{B}\circ \phi_p\right) \right) \varphi\circ \phi^{-1}_s.
\]
Then, we obtain by Riesz's theorem that $\widetilde \alpha\cdot D P\psi\in L^2(U)$ and that
\[
		\|\tilde \alpha\cdot DP\psi\|_{L^2(U)}\leq C\left(\|\psi\|^2_{L^2(\Omega)} + \|\alpha\cdot D \psi\|^2_{L^2(\Omega)}\right)\,,
\]
where $C>0$ does not depend on $\psi$.
\\
\emph{Step \eqref{stepc}.}
Let $\varphi\in C^\infty_0(U)$, we have
\[
\begin{split}
	&\|-i\widetilde \alpha\cdot \nabla \varphi\|_{L ^2(U)}^2 
	=\braket{\varphi,(-i\widetilde{\alpha}\cdot \nabla)^2 \varphi}_{U}
 -\braket{\varphi,\,\dive(\widetilde{\alpha}) \left(\widetilde{\alpha}\cdot \nabla\varphi\right)}_{U\cap\Omega^c}
\end{split}
\]
and
\[
	\left(-i\widetilde{\alpha}\cdot \nabla\right)^2 
	= -\sum_{j,k=1}^3\widetilde{\alpha}_j\widetilde{\alpha}_k\pa^2_{jk} 
	+ \left(\widetilde{\alpha}_j\pa_j\widetilde{\alpha}_k\right)\pa_k\,.
\]
Let us define the matrix-valued function $A$ for all $\x\in U$ by
\[
	A(\x) = |\jac\, \phi_s(\x)| (\jac\, \phi_s(\x))^{-1}\chi_{U\cap\Omega^c}(\x) + 1_3\chi_{U\cap \Omega}(\x) = (a_{jk}(\x))_{jk}
\]
and denote by $A_j(\x)$ the $j$-th line of $A(\x).$
We get that, for all $\x\in U$,
\[
\begin{split}
	&
	\widetilde\alpha_j(\x)\widetilde \alpha_k(\x)
	 = \mathcal{B}\circ \phi_p
	 \left(a_{j1}\alpha_1 + a_{j2}\alpha_2 + a_{j3}\alpha_3\right)\left(a_{k1}\alpha_1 + a_{k2}\alpha_2 + a_{k3}\alpha_3\right)
	 \mathcal{B}\circ \phi_p
	 \\
	 &
	 \quad
	 =
	 \left(\sum_{l=1}^3a_{jl}a_{kl}\right)1_4
	 + \mathcal{B}\circ \phi_p
	 \left(
	 \sum_{1\leq l<s\leq 3}\alpha_l\alpha_s\left(a_{jl}a_{ks}-a_{js}a_{kl}\right) 
	\right)
	 \mathcal{B}\circ \phi_p
\end{split}
\]
and
\[
	\sum_{j,k=1}^3\widetilde{\alpha}_j\widetilde{\alpha}_k\pa^2_{jk}
	=
	1_4\sum_{j,k=1}^3 A_jA_k^T\pa^2_{jk}\,.
\]
Since, $AA^T(\x) = 1_3$ for all $\x\in U\cap \partial \Omega $, we get that $\x\mapsto AA^T(\x)$ is a Lipschitzian application on $U$ and
\[
	\sum_{j,k=1}^3\widetilde{\alpha}_j\widetilde{\alpha}_k\pa^2_{jk}
	 = 
	 1_4\, \dive\left(A A^T\nabla\right)
	 -1_4\, \sum_{j,k=1}^3 \left(\pa_jA A^T\right)\pa_k.
\]
Integrating by parts yields
\begin{align*}
	\|-i\widetilde \alpha\cdot \nabla \varphi\|_{L ^2(U)}^2&\geq \|A^T\nabla \varphi\|^2_{L^2(U)}-C\|\varphi\|_{L^2(U)}\|\nabla\varphi\|_{L^2(U)}\\
										&\geq c\|\nabla \varphi\|^2_{L^2(U)}-C\|\varphi\|_{L^2(U)}\|\nabla\varphi\|_{L^2(U)}\,,
\end{align*}
where
\[
	c = \min\{ \inf \, \mathsf{sp} (AA^T(\x)), \,\x\in U\}\,.
\]
Note that $c>0$ by \eqref{eq:detjac}. This ensures that the $H^1$-norm and the $\|\cdot\|_{\mathcal{V}}$-norm are equivalent on $C^\infty_0(U)$.
\\
\emph{Step \eqref{stepd}.}
Let $v\in \mathcal{V}$ and $(\rho_\eps)_\eps$ a mollifier defined for $\x\in \R^3$ by
\[
	\rho_\eps(\x) = \frac{1}{\eps^3}\rho_1\left(\frac{\x}{\eps}\right)\,,
\]
where $\rho_1\in C^\infty_0(\R^3)$, $\supp\,\rho_1\subset B(0,1)$, $\rho_1\geq 0$ and $\|\rho_1\|_{L^1} = 1$.
Let us define $v_\eps = v* \rho_\eps$ for any $\eps>0$. There exists $\eps_0>0$ such that for all $\eps\in(0,\eps_0]$, the function $v_\eps$ belongs to $C^\infty_0(U)$. 
Let us temporarily admit that there exists $C$ independent of $v$  and $\eps$ such that
\begin{equation}\label{eq:density}
	\begin{split}
		&\|v_\eps\|_{\mathcal{V}}\leq C\|v\|_{\mathcal{V}}.
	\end{split}
\end{equation}
Then, Step \eqref{stepc} and the fact that $v_\eps$ converges to $v$ in $L^2(U)$ ensure that $\mathcal{V}\subset H^1_0(U)$ and the result follows.

It remains to prove \eqref{eq:density}.
There exists a constant $C>0$ such that
\[
	\|v_\eps\|_{L^2}\leq C\|v\|_{L^2}
\]
and
\[
	\begin{split}
		\|\widetilde{\alpha}\cdot D v_\eps\|_{L^2}
		&
		\leq 
		\|\widetilde{\alpha}\cdot \nabla v_\eps -\left(\widetilde{\alpha}\cdot \nabla v\right)*\rho_\eps\|_{L^2} 
		+ \|\left(\widetilde{\alpha}\cdot \nabla v\right)*\rho_\eps\|_{L^2}
		\\
		&
		\leq
		\|\widetilde{\alpha}\cdot \nabla v_\eps -\left(\widetilde{\alpha}\cdot \nabla v\right)*\rho_\eps\|_{L^2} 
		+ C\|\widetilde{\alpha}\cdot \nabla v\|_{L^2}\,.
	\end{split}
\]
By integration by parts, we get, for $\x\in U$,
\[\begin{split}
	&
	\widetilde{\alpha}\cdot \nabla v_\eps(\x) -\left(\widetilde{\alpha}\cdot \nabla v\right)*\rho_\eps(\x)
	\\
	&
	=
	\int_{\R^3}\widetilde\alpha(\x)\cdot \left(v(\y)\nabla\rho_\eps(\x-\y)\right) \dx\y
	-\int_{\R^3}\widetilde\alpha(\y)\cdot \nabla v(\y)\rho_\eps(\x-\y)
	\dx\y
	\\
	&
	=
	\int_{\R^3}\left(\widetilde\alpha(\x)-\widetilde\alpha(\y)\right)\cdot \left(v(\y)\nabla\rho_\eps(\x-\y)\right) \dx\y
	+\int_{\R^3}\left(\dive\,\widetilde\alpha(\y)\right)v(\y)\rho_\eps(\x-\y)
	\dx\y\,,
\end{split}\]
and by a change of variable
\[\begin{split}
	&
	\int_{\R^3}\left(\widetilde\alpha(\x)-\widetilde\alpha(\y)\right)\cdot \left(v(\y)\nabla\rho_\eps(\x-\y)\right) \dx\y
	\\
	&
	 = 
	 \int_{\R^3}\frac{\widetilde\alpha(\x)-\widetilde\alpha(\x-\eps \z)}{\eps}\cdot \left(v(\x-\eps \z)\nabla\rho_1(\z)\right) \dx\z\,.
\end{split}\]
Since $\widetilde\alpha$ is Lipschitzian, we get that
\[
	\left\|\int_{\R^3}\frac{\widetilde\alpha(\cdot)-\widetilde\alpha(\cdot-\eps \z)}{\eps}\cdot \left(v(\cdot-\eps \z)\nabla\rho_1(\z)\right) \dx\z\right\|_{L^2}\leq C\|v\|_{L^2}\||\cdot||\nabla \rho_1(\cdot)|\|_{L^1}\,,
\]
and
\[
	\left\|\int_{\R^3}\left(\dive\,\widetilde\alpha(\y)\right)v(\y)\rho_\eps(\cdot-\y)
	\dx\y\right\|_{L^2}\leq C\|v\|_{L^2}\,,
\]
so that \eqref{eq:density} follows.
 \end{proof}

Now we can end the proof of \eqref{eq.HsH}. Thanks to Proposition \ref{prop.extension}, the set $\D(H^\star)$ is included in $H^1(\Omega)$. Hence, for any $\psi\in \D(H^\star)$, the trace of $\psi$ on the set $\pa\Omega$ is well-defined and belongs to $H^{1/2}(\pa\Omega)$. 
	By the definition of $\D(H^\star)$ and an integration by parts, we obtain that, for any $\varphi\in \D(H)$,
	\[\begin{split}
		0 = \braket{\psi,H\varphi}_{\Omega} - \braket{H\psi,\varphi}_{\Omega} = \braket{\psi,-i\alpha\cdot n\varphi}_{\pa\Omega}= \braket{\beta\psi,\varphi}_{\pa\Omega}\,.
	\end{split}\]
	Hence, we have, for almost any $s\in \pa\Omega$,
	\[
		\beta \psi(s)\in \ker(\mathcal{B} - 1_4)^\perp = \ker(\mathcal{B} + 1_4)\,,
	\]
	so that
	\[
		\psi(s)\in \ker(\mathcal{B} - 1_4)\,,
	\]
	and we get \eqref{eq.HsH}.

\subsection{Proof of Point \ref{eq:theo14} in Theorem \ref{theo:selfadj}}

In the following lines, we assume that $\psi\in \D(H)$. First we expand the square to get
\[\|H\psi\|^2_{L^2(\Omega)} = \braket{\alpha\cdot D \psi,\alpha\cdot D \psi}_{\Omega} + m^2\braket{\beta\psi,\beta\psi}_{\Omega} + 2m\RE\braket{\beta\psi,\alpha\cdot D \psi}_{\Omega}\,.\]
Then we use \eqref{eq.IBP0} with $\varphi=\beta\psi$ and we find, by using that $\alpha$ anticommutes with $\beta$,
\[2\RE\braket{\beta\psi,\alpha\cdot D \psi}_{\Omega}=\braket{i\alpha\cdot\n\beta\psi,\psi,}_{\partial\Omega}=\braket{-i\beta\alpha\cdot\n\psi,\psi}_{\partial\Omega}=\|\psi\|^2_{L^2(\partial\Omega)}\,.\]
It remains to use that $\beta$ is unitary to deduce
\begin{equation}\label{eq.square}
\|H\psi\|^2_{L^2(\Omega)}= \|\alpha\cdot D \psi\|_{L^2(\Omega)}^2 + m^2\|\psi\|_{L^2(\Omega)}^2 + m\|\psi\|_{L^2(\pa\Omega)}^2\,.
\end{equation}
Assume moreover that $\psi\in H^2(\Omega)$. 
Then, we again use the Green-Riemann formula \eqref{eq.IBP0} and we have
\[ \braket{\alpha\cdot D \psi,\alpha\cdot D \psi}_{\Omega}=\braket{ \psi,(\alpha\cdot D)^2 \psi}_{\Omega}+\braket{(-i\alpha\cdot\n)\psi,\alpha\cdot D\psi}_{\partial\Omega}\,,\]
and thus, by noticing that $(\alpha\cdot D)^2=1_{4}D^2$, we find, by another integration by parts:
\[ \braket{\alpha\cdot D \psi,\alpha\cdot D \psi}_{\Omega}=\braket{ D\psi,D \psi}_{\Omega}+i\braket{\psi,\left((\alpha\cdot\n)(\alpha\cdot D)-(\n\cdot D)\right)\psi}_{\partial\Omega}\,.\]
Since $H^2(\Omega)$ is dense in $H^1(\Omega)$, we get that this formula holds for any $u\in \D(H)$.
We shall now investigate the boundary term by using the first algebraic relation in \eqref{eq:multalp}:
\[
	\begin{split}
	&
	i\braket{\psi,\left((\alpha\cdot\n)(\alpha\cdot D)-(\n\cdot D)\right)\psi}_{\partial\Omega}=-\braket{\psi,\gamma_{5}\alpha\cdot(\n\times D)\psi}_{\partial\Omega}
	\\
	&
	\qquad
	=-\braket{\gamma_{5}\psi,\alpha\cdot(\n\times D)\psi}_{\partial\Omega}\,.
	\end{split}	
\]
It remains to investigate the term $\braket{\gamma_{5}\psi,\alpha\cdot(\n\times D)\psi}_{\partial\Omega}$.
Since $\psi$ belongs to $\D(H)$, we have	
\[\braket{\gamma_{5}\psi,\alpha\cdot(\n\times D)\psi}_{\partial\Omega}=\braket{\gamma_{5}\psi,[\alpha\cdot(\n\times D),\mathcal{B}]\psi}_{\partial\Omega}+\braket{\gamma_{5}\psi,\mathcal{B}\alpha\cdot(\n\times D)\psi}_{\partial\Omega}\,,\]
and, since $\mathcal{B}$ is a symmetric operator, we get
\[\braket{\gamma_{5}\psi,\mathcal{B}\alpha\cdot(\n\times D)\psi}_{\partial\Omega}=\braket{\mathcal{B}\gamma_{5}\psi,\alpha\cdot(\n\times D)\psi}_{\partial\Omega}=-\braket{\gamma_{5}\mathcal{B}\psi,\alpha\cdot(\n\times D)\psi}_{\partial\Omega}\,.\]
We deduce that
\[\braket{\gamma_{5}\psi,\alpha\cdot(\n\times D)\psi}_{\partial\Omega}=\frac{1}{2}\braket{\gamma_{5}\psi,[\alpha\cdot(\n\times D),\mathcal{B}]\psi}_{\partial\Omega}\,,\]
and, with Lemma \ref{eq.anticurv}, we get
\[i\braket{\psi,\left((\alpha\cdot\n)(\alpha\cdot D)-(\n\cdot D)\right)\psi}_{\partial\Omega}=-\frac{1}{2}\braket{\gamma_{5}\psi,-\kappa\gamma_{5}\psi}_{\partial\Omega}\,,\]
and the conclusion follows.

\section{Large positive mass}\label{sec.3}
This section is devoted to the proofs of Theorems \ref{theo.positive} and \ref{theo.positive2}. For that purpose, one will work with the square of the Dirac operator $H^2$ appearing in Theorem \ref{theo:selfadj} and determine the asymptotic expansions of its lowest eigenvalues.

For $m>0$ and $\psi\in\mathsf{D}  = \{\psi \in H^1(\Omega,\C^4),\;\psi\in\ker\left(\mathcal{B}-1_{4}\right) \mbox{ on }\Gamma\}$, we let
\[Q_{m}(\psi)=\|\nabla\psi\|^2+\int_{\Gamma}\left(m+\frac{\kappa}{2}\right)|\psi|^2  \dx\Gamma\,.\]
In addition, we also define, for $\psi\in H^1_{0}(\Omega,\C^4)$,
\[Q_{\infty}(\psi)=\|\nabla\psi\|^2\,.\]
Let us denote by $(\lambda_j(\mathcal{Q}_m))_{j\geq 1}$ and $(\lambda_j(\mathcal{Q}_\infty))_{j\geq 1}$, the ordered sequence of  eigenvalues related to the operators associated with the quadratic forms $\mathcal{Q}_m$ and $\mathcal{Q}_\infty$. There respective $L^2$-normalized eigenfunctions are denoted by that $\psi_{j,m}$ and $\psi_{j,\infty}$.
\subsection{First non-trivial term in the asymptotic expansion}
Theorem \ref{theo.positive} is a consequence of the following proposition and of Theorem \ref{theo:selfadj}.
\begin{proposition}
For all $j\geq 1$, we have
\[\lim_{m\to+\infty}\lambda_{j}(Q_{m})=\lambda_{j}(Q_{\infty})\,.\]
\end{proposition}
\begin{proof}
Since $H^1_{0}(\Omega,\C^4)\subset \mathsf{D}$, we have, for all $n\geq 1$,
\[\lambda_{n}\left(Q_{m}\right)\leq\lambda_{n}\left(Q_{\infty}\right)\,.\]
Let us fix $N\geq 1$ and consider an orthonormal family $(\psi_{j,m})_{1\leq j\leq N}$ such that $\psi_{j,m}$ is an eigenfunction of the operator related to $Q_{m}$ and associated with its $j$-th eigenvalue. We set
\[E_{N}(m)=\mathsf{span}\,(\psi_{j,m})_{1\leq j\leq N}\,. \]
We easily get that, for all $\psi\in E_{N}(m)$,
\[Q_{m}(\psi)\leq\lambda_{N}(Q_{m})\|\psi\|^2\leq\lambda_{N}(Q_{\infty})\|\psi\|^2\,.\]
Let us first prove that $\lambda_{1}(Q_{m})$ converges towards $\lambda_{1}(Q_{\infty})$. 
For that purpose, let us establish that the only accumulation point of $(\lambda_{1}(Q_{m}))_{m\geq 0}$ is $\lambda_{1}(Q_{\infty})$. Since $(\psi_{1,m})$ is bounded in $H^1(\Omega)$, we may assume, up to a subsequence extraction, that $\psi_{1,m}$ converges weakly to $\psi_{1,\infty}\in H^1(\Omega)$. But, we have
\[\int_{\Gamma}|\psi_{1,m}|^2\dx\Gamma=\mathcal{O}(m^{-1})\,,\]
and by the Fatou lemma, $\psi_{1,\infty}=0$ on $\Gamma$ so that $\psi_{1,\infty}\in H^1_{0}(\Omega)$. Then, we get
\[\lambda_{1}(Q_{\infty})\geq\lim_{m\to+\infty}\lambda_{1}(Q_{m})\geq\liminf_{m\to+\infty}\|\nabla\psi_{1,m}\|^2\geq\|\nabla\psi_{1,\infty}\|^2\geq \lambda_{1}(Q_{\infty})\,.\]
We deduce that $\psi_{1,\infty}$ is an eigenfunction of the Dirichlet Laplacian associated with $\lambda_{1}(Q_{\infty})$.
Therefore, we have the convergence result for the first eigenvalue. We also get that $(\psi_{1,m})$ converges to $\psi_{1,\infty}$ strongly in $H^1(\Omega)$.

Let us now proceed by induction. Let $N\geq 1$. Assume that, for all $j\in\{1,\dots,N\}$, $(\lambda_{j}(Q_{m}))$ converges to $\lambda_{j}(Q_{\infty})$ and that, up to a subsequence extraction, $(\psi_{j,m})$ converges to $\psi_{j,\infty}$, an eigenfunction associated with $\lambda_{j}(Q_{\infty})$. As above, we may assume that $(\psi_{N+1,m})$ weakly converges to some $\psi_{N+1,\infty}\in H^1(\Omega)$ and that its trace on $\Gamma$ is zero. We also get, by convergence in $L^2(\Omega)$, that
\[\psi_{N+1,\infty}\in\left(\underset{1\leq j\leq N}{\mathsf{span}}\, \psi_{j,\infty}\right)^{\perp}\,.\]
By the min-max principle, it follows that
\[\lambda_{N+1}(Q_{\infty})\geq\lim_{m\to+\infty}\lambda_{N+1}(Q_{m})\geq\liminf_{m\to+\infty}\|\nabla\psi_{N+1,m}\|^2\geq\|\nabla\psi_{N+1,\infty}\|^2\geq \lambda_{N+1}(Q_{\infty})\,.\]
From these last inequalities, we infer that $\psi_{N+1,\infty}$ is an eigenfunction of the Dirichlet Laplacian associated with $\lambda_{N+1}(Q_{\infty})$, that $(\lambda_{N+1}(\mathcal{Q}_m))$ converges to $(\lambda_{N+1}(\mathcal{Q}_\infty))$ and $(\psi_{N+1,m})$ converges strongly in $H^1(\Omega)$ to $\psi_{N+1,\infty}$.
\end{proof}
\subsection{Asymptotic expansion of the  first eigenvalue }
The following lemma will be used in the proof of Theorem \ref{theo:selfadj}.
\begin{lemma}\label{lem.norm-der-n}
Let $u\in H^1_0(\Omega,\C)$ be an $L^2$-normalized eigenfunction of the Dirichlet Laplacian on $\Omega$. Then
 \[\int_\Gamma |\pa_\n u|^2\n \dx\Gamma = 0\,. \]
\end{lemma}
\begin{proof}
We have $\nabla u = \left(\pa_\n u\right)\n$ so that by integration by parts, we get
	 \[
	 	\begin{split}
			&\int_\Gamma |\pa_\n u|^2\n \dx\Gamma = \int_\Gamma |\nabla u|^2\n \dx\Gamma = \int_\Omega\nabla |\nabla u|^2 \dx\x  = \left(\int_\Omega 2\nabla u\cdot \nabla \pa_k u \dx\x\right)_{k=1,2,3}
			\\
			& \quad =  2\left(\int_\Omega(-\Delta u) \pa_k u\dx\x\ + \int_\Gamma \pa_\n u \pa_k u\dx\Gamma\right)_{k=1,2,3}
			\\
			& \quad =  2\left(\lambda_j(\mathcal{Q}_\infty)/2\int_\Omega\pa_k |u|^2\dx \x + \int_\Gamma \pa_\n u \pa_k u\dx\Gamma\right)_{k=1,2,3} 
			\\
			&\quad = 2\int_\Gamma \pa_\n u \nabla u\dx\Gamma 
			 = 2\int_\Gamma |\pa_\n u|^2\n \dx\Gamma\,,
		\end{split}
	 \]
	 and the conclusion follows.
\end{proof}

Theorem \ref{theo.positive2} is a consequence of the following proposition and of Theorem \ref{theo:selfadj}.
\begin{proposition}
	Let $u_1\in H^1_0(\Omega)$ be an $L^2$-normalized eigenfunction of the Dirichlet Laplacian associated with its lowest eigenvalue $\lambda_1(\mathcal{Q}_\infty)$. 
	We have that
	\[
		\lambda_{1}(\mathcal{Q}_m) = \lambda_1(\mathcal{Q}_\infty) -\frac{1}{2m}\int_{\Gamma}|\pa_\n u_1|^2\dx\Gamma+ \mathcal{O}(m^{-2})\,.	
	\]
\end{proposition}
\begin{remark}
	In the case of the Robin Laplacian, we obtain
	\[
		\lambda_{1}^{\mathsf{Rob}}(\mathcal{Q}_m) = \lambda_1(\mathcal{Q}_\infty) -\frac{1}{m}\int_{\Gamma}|\pa_\n u_1|^2\dx\Gamma + \mathcal{O}(m^{-2})
	\]
	and we recover asymptotically the fact that $\lambda_1^{\mathsf{Rob}}(\mathcal{Q}_m)\leq \lambda_1(\mathcal{Q}_m)$. 
\end{remark}
\begin{proof}
	The proof of this result is divided into three steps: 
	\begin{enumerate}[(a)]
		\item \label{eq:stepNrl1} we perform a formal study of the asymptotic expansion of $\lambda_1(\mathcal{Q}_m)$,
		\item \label{eq:stepNrl2} we build rigorously a test function based on Step \eqref{eq:stepNrl1} to get the upper bound,
		\item \label{eq:stepNrl3} we study the lower bound.
	\end{enumerate}
	 \emph{Step} \eqref{eq:stepNrl1}. We look for quasi-eigenvalues and quasi-eigenfunctions in the form
	 \[
	 	\begin{split}
			&\lambda^\mathsf{app}_1(\mathcal{Q}_m) = \lambda_1(\mathcal{Q}_{\infty}) + \frac{\lambda}{m} + \mathcal{O}(m^{-2})\,,\\
			& \psi^{\mathsf{app}}_{1,m} = \psi_{1,\infty} + m^{-1}\varphi + \mathcal{O}(m^{-2})\,,
		\end{split}
	 \]
	 where $\lambda$ and $\varphi$ are unknown.

	 We recall that $\psi_{1,m}$ and $\psi_{1,\infty}$ satisfy
	 \[
	 	\begin{split}
			&-\Delta\psi_{1,m} = \lambda_1(\mathcal{Q}_{m})\psi_{1,m}, \mbox{ on } \Omega\,,\\
			&\psi_{1,m} \in \ker (\mathcal{B}-1_4), \mbox{ on } \Gamma\,,\\
			&(\pa_\n + \kappa/2 + m)\psi_{1,m}\in \ker(\mathcal{B}+1_4)\,, \mbox{ on } \Gamma\,.
		\end{split}
	 \]
	and
	 \[
	 	\begin{split}
			&-\Delta \psi_{1,\infty} = \lambda_1(\mathcal{Q}_{\infty}) \psi_{1,\infty}\,, \mbox{ on }\Omega,\\
			&\psi_{1,\infty} = 0\,, \mbox{ on } \Gamma\,.
		\end{split}
	 \]
	 Then, we want that
	 \begin{equation}\label{eq:varphiAsymptDir1}
	 	\begin{split}
			&\left(-\Delta -\lambda_1(\mathcal{Q}_{\infty}) \right) \varphi = \lambda \psi_{1,\infty}, \mbox{ on }\Omega\,,\\
			&\varphi\in \ker (\mathcal{B}-1_4), \mbox{ on }  \Gamma\,,\\
			&\pa_{\n}\psi_{1,\infty} + \varphi\in \ker (\mathcal{B}+1_4), \mbox{ on }  \Gamma\,.
		\end{split}
	 \end{equation}
	 Denoting for all $s\in  \Gamma$, $P_+(s) = \frac{1-\mathcal{B}(s)}{2}$, the orthogonal projection on $\ker (\mathcal{B}-1_4)$, we get that 
	 \[
	 	\begin{split}
			&0 = P_+\left(\pa_{\n}\psi_{1,\infty} + \varphi\right) = P_+\pa_{\n}\psi_{1,\infty} + \varphi\,.
		\end{split}
	 \]
	 Taking the scalar product of equation \eqref{eq:varphiAsymptDir1} with $\psi_{1,\infty}$ and integrating by parts twice, we obtain that 
	 \[
	 	\lambda = -\|P_+\pa_\n\psi_{1,\infty}\|^2_{L^2(\Gamma)}
	 \]
	 and
	  \begin{equation}\label{eq:varphiAsymptDir2}
	 	\begin{split}
			&\left(-\Delta -\lambda_1(\mathcal{Q}_{\infty}) \right) \varphi= \lambda \psi_{1,\infty}, \mbox{ on }\Omega\,,\\
			& \varphi = -P_+\pa_{\n}\psi_{1,\infty}, \mbox{ on }\Gamma\,.
		\end{split}
	 \end{equation}
	 Let us now consider $\lambda$. Note, that for all eigenfunction $\psi_{1,\infty}$ of the Dirichlet Laplacian in $L^2(\Omega,\C^4)$ associated with its lowest eigenvalue $\lambda_1(\mathcal{Q}_{\infty})$, there exists $a\in \C^4$ such that $|a|=1$ and $\psi_{1,\infty} = au_{1}$. Then, we have
	 \[
	 	\begin{split}
			\lambda &= - \frac{1}{2}\int_{\Gamma}|\pa_\n u_1|^2\left(1 + \braket{a,\mathcal{B}a}\right)\dx\Gamma\\
			& = - \frac{1}{2}\int_{\Gamma}|\pa_\n u_1|^2\dx\Gamma - \frac{1}{2}\braket{a,-i\beta \alpha\cdot\left(\int_\Gamma|\pa_\n u_1|^2\n \dx\Gamma\right)a}\,.
		\end{split}
	 \]
	 With Lemma \ref{lem.norm-der-n}, we obtain that 
	\[
		\lambda= - \frac{1}{2}\int_{\Gamma}|\pa_\n u_1|^2\dx\Gamma.
	\]
	 %
	 \\
	 \emph{Step} \eqref{eq:stepNrl2}.
	 Let $\psi_{1,\infty} = a u_1$ be an eigenfunction of the Dirichlet Laplacian associated with $\lambda_1(\mathcal{Q}_\infty)$ and $w\in H^2(\Omega)$ be such that $w = -P_+\pa_\n \psi_{1,\infty}$. Let us study the existence of a solution $\varphi_1$ of equation \eqref{eq:varphiAsymptDir2}. We denote by $(-\Delta)^{-1}$ the inverse of the Dirichlet Laplacian and $v = \varphi_1-w$ so that
	 \[
	 	\begin{split}
			\left(\mathsf{Id} - \lambda_{1}(\mathcal{Q}_\infty)(-\Delta)^{-1}\right)v = (-\Delta)^{-1}\lambda\psi_{1,\infty} - (-\Delta)^{-1}\left(-\Delta -  \lambda_{1}(\mathcal{Q}_\infty)\right)w\,.
		\end{split}
	 \]
	 By the Fredholm alternative, there exists such a function $v$ if and only if 
	 \[
	 	 (-\Delta)^{-1}\left(\lambda\psi_{1,\infty} - \left(-\Delta -  \lambda_{1}(\mathcal{Q}_\infty)\right)w\right) \in \ker\left(\mathsf{Id} -  \lambda_{1}(\mathcal{Q}_\infty)(-\Delta)^{-1}\right)^\perp\,.
	 \]
	 Let $\psi\in \ker\left(\mathsf{Id} -  \lambda_{1}(\mathcal{Q}_\infty)(-\Delta)^{-1}\right)^\perp$. We have by integrations by parts that
	 \[
	 	\begin{split}
			&\braket{\psi,(-\Delta)^{-1}\left(\lambda\psi_{1,\infty} - \left(-\Delta -  \lambda_{1}(\mathcal{Q}_\infty)\right)w\right) }_\Omega
			\\
			&
			\quad
			={ \lambda_{1}(\mathcal{Q}_\infty)}^{-1}\left(\braket{\psi,\lambda_1\psi_{1,\infty}}_\Omega - \braket{\left(-\Delta -  \lambda_{1}(\mathcal{Q}_\infty)\right)\psi,w }_\Omega  -  \braket{\pa_\n \psi, w}_\Gamma\right)
			\\
			&
			\quad
			= \lambda_{1}(\mathcal{Q}_\infty)^{-1}\left(\braket{\psi,\lambda_1\psi_{1,\infty}}_\Omega  +  \braket{\pa_\n \psi, P_+\pa_{\n}\psi_{1,\infty}}_\Gamma\right)\,.
		\end{split}
	 \]
	 Hence, we get
	 \[
	 	\begin{split}
			0 = \braket{\psi_{1,\infty},(-\Delta)^{-1}\left(\lambda\psi_{1,\infty} - \left(-\Delta - \lambda_{1,\infty}\right)w\right) }_\Omega
		\end{split}
	 \]
	 provided that
	 \begin{equation}\label{eq:lambda1}
	 	\lambda = - \int_\Gamma|P_+\left(\pa_{\n}\psi_{1,\infty}\right)|^2\dx\Gamma\,.
	 \end{equation}
	 Let $a,b\in\C^4$ be such that $\braket{a,b} = 0$, $|a| = |b|=1$, $\psi_{1,\infty} = au_1$ and $\psi = bu_1$. We have
	 \[
	 	\begin{split}
			&0 = \braket{\psi,(-\Delta)^{-1}\left(\lambda_1\psi_{1,\infty} - \left(-\Delta -  \lambda_{1}(\mathcal{Q}_\infty)\right)w\right) }_\Omega
		\end{split}
	 \]
	 since
	 \[
	 	0 = \braket{\pa_\n \psi, P_+\pa_{\n}\psi_{1,\infty}}_\Gamma 
		= \frac{1}{2}\braket{b,-i\beta \alpha\cdot \left(
			\int_\Gamma |\pa_\n u_0|^2 \n \dx\Gamma
		\right)a}\,.
	 \]
	 Hence, assuming that \eqref{eq:lambda1} is true, we get that system \eqref{eq:varphiAsymptDir2} has a solution $\varphi_1$. $\psi_{1,\infty} + m^{-1}\varphi_1$ can be used as a test function and we have 
	 \[
	 	\begin{split}
		\mathcal{Q}_m(\psi_{1,\infty} + m^{-1}\varphi_1) 
		&
		= \lambda_1(\mathcal{Q}_{\infty}) 
		+ m^{-1}\left(2\Re\braket{\nabla \psi_{1,\infty},\nabla\varphi_1}_{\Omega}
		+\int_{\Gamma}|\varphi_1|^2\dx\Gamma\right)
		+ \mathcal{O}(m^{-2})
		\\
		&
		 = \lambda_1(\mathcal{Q}_{\infty}) \|\psi_{1,m}\|^2_{L^2}
		- m^{-1}\int_\Gamma|P_+\left(\pa_{\n}\psi_{1,\infty}\right)|^2\dx\Gamma
		+ \mathcal{O}(m^{-2})
		\end{split}
	 \]
	 so that
	 \begin{equation}\label{eq:bornesupNL}
		\lambda_1(\mathcal{Q}_m)\leq \lambda_1(\mathcal{Q}_{\infty})- m^{-1}\int_\Gamma|P_+\left(\pa_{\n}\psi_{1,\infty}\right)|^2\dx\Gamma
		+ \mathcal{O}(m^{-2})\,.
	 \end{equation}
	\\
	\emph{Step} \eqref{eq:stepNrl3}.
	 Let us now study the lower bound. The sequence $(\psi_{1,m})$ is uniformly bounded in $H^1(\Omega)$.
	 We extract a subsequence $(m_k)_{k\in\N}$ such that
	\[
		\underset{m\to +\infty}{\lim \inf}	\;m\left(\lambda_{1,m}-\lambda_{1,\infty}\right) = \lim_{k\to +\infty} m_k\left(\lambda_{1,m_k}-\lambda_{1,\infty}\right)
	\]
	and $(\psi_{1,m_k})_{k\in\N}$ converges strongly in $H^1(\Omega)$ to $\psi_{1,\infty}\in H^1_0(\Omega)$ and $(\pa_\n\psi_{1,m_k})$ converges to $(\pa_\n\psi_{1,\infty})$ in $H^{-1/2}(\Gamma)$.
	Integrating by parts yields
	\begin{equation}\label{eq:extract}
		\begin{split}
			(\lambda_{1,m_k}-\lambda_{1,\infty})\braket{\psi_{1,m_k},\psi_{1,\infty}}_\Omega = -{m_k}^{-1}\braket{(\kappa/2m_k + 1)^{-1}\pa_\n\psi_{m_k,\infty},P_+\pa_\n \psi_{1,\infty}}_\Gamma\,,
		\end{split}
	\end{equation}
	so that by Step \eqref{eq:stepNrl1},
	\[
		\underset{m\to +\infty}{\lim \inf}	\;m\left(\lambda_{1,m}-\lambda_{1,\infty}\right) = - \|P_+\pa_\n \psi_{1,\infty}\|^2_{L^2(\Gamma)} \geq - \frac{1}{2}\int_{\Gamma}|\pa_\n u_1|^2\dx\Gamma 
	\]
	and the result follows.
\end{proof}

\section{Large negative mass: main steps in the proof of Theorem \ref{theo.bis}}\label{sec.4}
In this section, we study the non-relativistic limit $m\to +\infty$ of the nonnegative eigenvalues of the MIT bag Dirac operator $H_{-m}^\Omega$. 
For the sake of readability, we present the main ingredients used in the proof of Theorem \ref{theo.bis}. Part of the ideas are related to recent results about the semiclassical Robin Laplacians (see \cite[Section 7]{HKR15}, \cite{HK-tams} and \cite{KKR16}). The detailed proofs will be given in Section \ref{sec.5}.

\subsection{Semiclassical reformulation and boundary localization}\label{sec.semi}
The main objective of this section is to get boundary localization results of Agmon type.
For that purpose, we will rather consider $\left(H_{-m}^\Omega\right)^2$ and introduce the semiclassical parameter
\[h=m^{-2}\to 0\,.\]

\subsubsection{The semiclassical operator}
In order to lighten the presentation, it will also be more convenient to work with the following operator
\begin{equation}\label{eq.Lh}
	\mathscr{L}_h = h^2((H^\Omega_{-m})^2-m^2 1_4)\,,
\end{equation}
whose domain is given by
\begin{multline*}
\D(\mathscr{L}_h) = \D((H^\Omega_{-m})^2)\\
	= \left\{
	\psi\in H^2(\Omega)\;:\;\psi\in \ker(\mathcal{B}-1_4),\;\left(\partial_{\n} + \frac{\kappa}{2} -h^{-\frac{1}{2}}\right)\psi\in \ker(\mathcal{B}+1_4),\mbox{ on } \Gamma
	\right\}.
\end{multline*}
The associated quadratic $\mathscr{Q}_h$ form is defined by 
\begin{equation}\label{eq:sesformsc}
\forall\psi\in\D(\mathscr{Q}_h)\,,\qquad	\mathscr{Q}_h(\psi) = h^2\|\nabla\psi\|^2_{L^2(\Omega)}+\int_{\Gamma}\left(\frac{\kappa}{2}h^2-h^{\frac{3}{2}}\right)|\psi|^2\dx \Gamma\,,
\end{equation}
where
\[
	\D(\mathscr{Q}_h) = \D(H^\Omega_{-m}) = \left\{
	\psi\in H^1(\Omega)\;:\;\psi\in \ker(\mathcal{B}-1_4)\mbox{ on }\Gamma
	\right\}\,.
\]
In other words, the operator $\mathscr{L}_{h}$ is the semiclassical Laplacian with combined MIT bag condition \emph{and} Robin condition on the boundary.
\subsubsection{Relations between the eigenvalues of $\mathscr{L}_h$ and $H^\Omega_{-m}$}\label{sec:scaleH1}
 Let us describe the relations between the spectra of our operators. Let us recall that the spectrum of $H^\Omega_{-m}$ is discrete, symmetric with respect to $0$ and with pair multiplicity. The spectrum of $H^\Omega_{-m}$ lying in $[-m, m]$ is given by 
\[\left\{\pm\sqrt{h^{-2}\lambda_{n}(h)+h^{-1}} \,:\,n\in\N\setminus\{0\}\,, -h\leq \lambda_{n}(h)\leq 0\right\}\,,\]
where $\lambda_n(h)$ denotes the $n$-th eigenvalue of $\mathscr{L}_{h}$.
Therefore, we shall focus on the study of the negative eigenvalues of $\mathscr{L}_{h}$.

\subsubsection{Localization estimates \textit{\`a la} Agmon} 
The estimates given in Proposition \ref{red.bound} are a consequence of the fact that the Laplacian is a non-negative operator.
\begin{proposition}\label{red.bound}
Let $\epsilon_0\in (0,1)$ and $\gamma\in(0,\sqrt{\eps_{0}})$. There exists $C>0$ such that for any $h\in(0,1]$, any eigenvalue $\lambda\leq -\eps_0h$ of $\mathscr{L}_h$ and any eigenfunction $\psi_h$ of $\mathscr{L}_h$ associated with $\lambda$, we have
\[
	\left\|\psi_h \exp\left({\frac{\gamma d(\cdot,\Gamma)}{h^{1/2}}}\right)\right\|^2_{L^2(\Omega)} + h^{-1}\left|\mathscr{Q}_h\left(\psi_h \exp\left({\frac{\gamma d(\cdot,\Gamma)}{h^{1/2}}}\right)\right)\right|\leq C\|\psi_h\|^2_{L^2(\Omega)}. 
\]
\end{proposition}

\subsection{The operator near the boundary}
Relying on Proposition \ref{red.bound}, we introduce the operator near the boundary. 
Given $\delta \in (0,\delta_0)$ (with $\delta_0>0$ small enough), we introduce the $\delta$-neighborhood of the boundary
\begin{equation}\label{eq.delta.neighbor}
\mathcal V_{\delta}=\{\x\in\Omega~:~{\rm dist}(\x,\Gamma)<\delta\}\,,
\end{equation}
and the quadratic form, defined on the variational space
\begin{multline*}
V_{\delta}=\Big\{u\in H^1(\mathcal V_{\delta})~:~u(\x)=0\,\quad\mbox{ for all } \x\in\Omega \mbox{ such that } {\rm dist}(\x,\Gamma)=\delta\\
 \mathrm{ and }\,\quad  \mathcal{B}u=u \mbox{ on }\Gamma\Big\}\,,
\end{multline*}
by the formula
\[\forall u\in V_{\delta}\,,\qquad\mathscr{Q}_{h}^{\{\delta\}}(u)=\int_{\mathcal{V}_{\delta}}|h\nabla u|^2 \dx\x+\int_{\Gamma}\left(\frac{\kappa}{2}h^2-h^{\frac{3}{2}}\right)|u|^2\dx \Gamma\,.\]
We denote  by $\mathscr{L}_{h}^{\{\delta\}}$ the corresponding operator.
\\
\subsubsection{The operator near the boundary in tubular coordinates.} 
Let $\iota$ be the canonical embedding of $\Gamma$ in $\R^3$ and $g$ the induced metrics on $\Gamma$. $(\Gamma,g)$ is a $\mathcal{C}^3$ Riemannian manifold, which we orientate according to the ambient space. Let us introduce the map $\Phi:\Gamma\times(0,\delta)\to\mathcal{V}_{\delta}$ defined by the formula
\begin{equation*}
\Phi(s,t)=\iota(s)-t\n(s)\,.
\end{equation*}
The transformation $\Phi$ is a $\mathcal{C}^3$ diffeomorphism for any $\delta\in(0,\delta_0)$ provided that $\delta_0$ is sufficiently small. The induced metrics on $\Gamma\times(0,\delta)$ is given by
\[G=g\circ (\mathsf{Id}-tL(s))^{2}+\dx t^2\,,\]
where $L(s)=d\n_{s}$ is the second fundamental form of the boundary at $s$.
Let us now describe how our MIT bag - Robin Laplacian is transformed under the change of coordinates. For all $u\in L^{2}(\mathcal V_{\delta})$, we define the pull-back function
\begin{equation}\label{eq:trans-st}
\widetilde u(s,t):= u(\Phi(s,t)).
\end{equation}
For all $u\in H^{1}(\mathcal{V}_{\delta})$, we have
\begin{equation}\label{eq:bc;n}
\int_{\mathcal{V}_{\delta}}|u|^{2}\dx\x=\int_{\Gamma\times(0,\delta)}|\widetilde u(s,t)|^{2}\,\tilde a\dx \Gamma \dx t\,,
\end{equation}
\begin{equation}\label{eq:bc;qf}
\int_{\mathcal{V}_{\delta}}|\nabla u|^{2}\dx\x= \int_{\Gamma\times(0,\delta)} \Big[\langle\nabla_{s} \widetilde u,\tilde g^{-1}\nabla_{s} \widetilde u\rangle +|\partial_{t}\widetilde u|^{2}\Big]\,\tilde a\dx \Gamma\dx t\,.
\end{equation}
where
\[\tilde g=\big(\mathsf{Id}-tL(s)\big)^2\,,\]
and $\tilde a(s,t)= |\tilde g(s,t)|^{\frac{1}{2}}$. Here $\langle\cdot,\cdot\rangle$ is the Euclidean scalar product and $\nabla_{s}$ is the differential on $\Gamma$ seen through the metrics $g$.
Since $L(s)\in \C^{2\times 2}$, we have the exact formula
\begin{equation}\label{eq.Taylor3}
	\tilde{a}(s,t) = 1-t\kappa(s) + t^2K(s)
\end{equation}
where
\[\kappa(s)=\mathsf{Tr}\, L(s)\, \mbox{ and }K(s) =  \det \, L(s).\]
The operator $\mathscr L^{\{\delta\}}_h$ is expressed in $(s,t)$ coordinates as
\[\widetilde{\mathscr L}^{\{\delta\}}_h=-h^2\tilde a^{-1}\nabla_s(\tilde a \tilde g^{-1}\nabla_s)-h^2\tilde a^{-1}\partial_t(\tilde a\partial_t)\,.\]
 In these coordinates, the Robin condition becomes
\[h^2\partial_tu=\left(\frac{\kappa}{2}h^2-h^{\frac 32}\right)u\quad{\rm on}\quad t=0\,.\]
We introduce, for $\delta \in (0,\delta_0)$,
\begin{equation}\label{eq:red-bnd}
\begin{aligned}
&\widetilde{\mathcal V}_\delta=\{(s,t)~:~s\in \Gamma\quad{\rm and}\quad 0<t< \delta\}\,, \\ 
& \widetilde V_\delta =\{u\in H^1(\widetilde{\mathcal V_\delta}),\,u(\cdot,0)\in\ker\left(\mathcal{B}-1_{4}\right),~u(\cdot,\delta)=0\}\,,\\
&\widetilde{\mathcal D}_\delta=\{u\in H^2(\widetilde{\mathcal V_{\delta}})\cap \widetilde V_\delta ~:~\partial_tu(\cdot,0)-\left(\frac{\kappa}{2}-h^{-\frac{1}{2}}\right)u(\cdot,0)\in \ker (\mathcal{B}+1_4)\}\,,\\
&\widetilde{\mathscr{Q}}_{h}^{\{\delta\}}(u)=\int_{\widetilde{\mathcal V_\delta}}\Big(h^2\langle\nabla_{s} u,\tilde g^{-1}\nabla_{s} u\rangle+|h\partial_tu|^2\Big)\tilde a \dx \Gamma\dx t+\int_{\Gamma}\left(\frac{\kappa}{2}h^2-h^{\frac 32}\right) |u(s,0)|^2 \dx \Gamma\,.
\end{aligned}
\end{equation}
The operator $\widetilde{\mathscr L}_h^{\{\delta\}}$ acts on $L^2(\widetilde{\mathcal V}_\delta,\tilde a\dx t\dx \Gamma)$.

Let us denote by $\lambda^{\{\delta\}}_n(h)$ the $n$-th eigenvalue of the corresponding operator $\widetilde{\mathscr L}_h^{\{\delta\}}$. Using smooth cut-off functions, the min-max principle and the Agmon estimates of Proposition \ref{red.bound}, it is then standard to deduce the following proposition (see \cite{H88}).
\begin{proposition}\label{prop:red-bnd}
Let $\epsilon_{0}\in(0,1)$ and $\gamma\in(0,\sqrt{\epsilon_{0}})$. There exist constants $C>0$, $h_0\in(0,1)$ such that, for all $h\in(0,h_0)$, $\delta\in(0,\delta_{0})$, $n\geq 1$ such that $\lambda_{n}(h)\leq-\epsilon_{0}h$,
\begin{equation}\label{complem}
\lambda_n(h)\leq \lambda^{\{\delta\}}_n(h)\leq \lambda_n(h)+C\exp\left(-\gamma  \delta h^{-\frac 12} \right)\,.
\end{equation}
\end{proposition}
In the following, it is sufficient to choose 
\begin{equation}\label{eq:deltacond}\delta=h^{\frac{1}{4}}\,.\end{equation}

\subsection{The rescaled MIT bag operator in boundary coordinates}\label{sec.rescaled}
Looking at the rate of convergence obtained in Proposition \ref{red.bound}, we perform a change of scale in the normal direction that allows us to see something at the limit. 
We introduce the rescaling
\[(s,\tau)=(s,h^{-\frac 12}t)\,,\]
the new semiclassical parameter $\hbar=h^{\frac 1 4}$ and the new weights
\begin{equation}\label{eq:Jac-a'}
\widehat a_\hbar(s,\tau)=\tilde a(s,h^{\frac{1}{2}}\tau)\,,\qquad \widehat g_\hbar(s,\tau)=\tilde g(s,h^{\frac{1}{2}}\tau)\,.
\end{equation}
We  also introduce the parameter 
\begin{equation}\label{eq:Tcond}
	T_\hbar = \delta h^{-\frac 12}  = h^{-\frac 14} = \hbar^{-1}
\end{equation} (see \eqref{eq:deltacond}).
We consider rather the operator
\begin{equation}\label{eq.Lh-hat}
\widehat{\mathscr{L}}_{\hbar}=h^{-1}\widetilde{\mathscr{L}}_{h}\,,
\end{equation}
acting on $L^2({\widehat{\mathcal V}}_{\hbar},\widehat a_\hbar \dx \Gamma \dx\tau)$ and expressed in the rescaled coordinates $(s,\tau)$. 

As in \eqref{eq:red-bnd}, we let
\begin{equation}\label{eq:dom-Lh-hat}
\begin{aligned}
&\widehat{\mathcal V}_{\hbar}=\{(s,\tau)~:~s\in \Gamma~{\rm and}~0<\tau<
 \hbar^{-1}  \}\,, \\
&\widehat V_{\hbar}=\{u\in H^1(\widehat{\mathcal V}_{\hbar};\widehat a_\hbar \dx \Gamma \dx\tau),\,u(\cdot,0)\in\ker\left(\mathcal{B}-1_{4}\right),~u(\cdot,\hbar^{-1})=0\}\,,\\
&\widehat{\mathcal D}_{\hbar}=\{u\in H^2(\widehat{\mathcal V}_{\hbar};\widehat a_\hbar \dx \Gamma \dx\tau)\cap \widehat V_{\hbar}~:~\partial_\tau u(\cdot,0)-\left(\frac{\kappa}{2}\hbar^2-1\right)u(\cdot,0) \in \ker (\mathcal{B}+1_4)\}\,,\\
&\widehat{\mathscr{Q}}_{\hbar}(u)=\int_{\widehat{\mathcal V}_{\hbar}}\Big(\hbar^4\langle\nabla_{s} u,\widehat g^{-1}_\hbar\nabla_{s} u\rangle+|\partial_\tau u|^2\Big)\widehat a_\hbar \dx \Gamma \dx\tau+\int_{\Gamma} \left(\frac{\kappa}{2}\hbar^2-1\right)|u(s,0)|^2\dx \Gamma\,,\\
&\widehat{\mathscr{L}}_\hbar=-\hbar^4 \widehat a^{-1}_\hbar\nabla_s(\widehat a_\hbar \widehat g^{-1}_\hbar\nabla_s)-\widehat a^{-1}_\hbar\partial_\tau\widehat a_\hbar\partial_\tau\,.
\end{aligned}
\end{equation} 
\subsection{Contribution of the normal variable}
Let us notice that the first order terms in \eqref{eq:dom-Lh-hat} are related to the normal variable.
Hence, we are naturally led to introduce the following quadratic form gathering all the terms acting in the normal direction: 
\begin{equation}\label{eq:dom-Lh-hat-first}
\begin{aligned}
	&\widehat V^1_\hbar=\{u\in L^2(\widehat{\mathcal V}_{\hbar};\widehat a_\hbar \dx \Gamma \dx\tau),\, \pa_\tau u\in L^2(\widehat{\mathcal V}_{\hbar};\widehat a_\hbar \dx \Gamma \dx\tau),\,\\
	&\qquad\qquad\qquad\qquad\qquad\qquad u(\cdot,0)\in\ker\left(\mathcal{B}-1_{4}\right),~u(\cdot,\hbar^{-1})=0\}\,,\\
	&\widehat{\mathscr{Q}_{\hbar}^{1 }}(u)=\int_{\Gamma}\Big(\int_0^{\hbar^{-1}}|\partial_\tau u|^2\widehat a_\hbar \dx\tau + \left(\frac{\kappa}{2}\hbar^2-1\right)|u(s,0)|^2\Big)\dx \Gamma\,.
\end{aligned}
\end{equation} 
The goal of this section is to study the lowest part of the spectrum of the operator $\widehat{\mathscr{L}}^1_\hbar$ associated with the quadratic form $\widehat{\mathscr{Q}_{\hbar}^{1 }}$. 
\subsubsection{Diagonalization of the boundary condition}\label{sec.diago-BC}
Without the gradient term in the $s$-direction appearing in $\widehat{\mathscr{Q}}_{\hbar}(u)$, the MIT bag boundary condition can be diagonalized for every $s\in \Gamma$.
Let us introduce for all $s\in \Gamma$, the unitary $4\times 4$ matrix
\[
	P_{\n}:=\frac{1}{\sqrt{2}}\begin{pmatrix}
1_{2}&i\sigma\cdot\n\\
i\sigma\cdot\n&1_{2}
\end{pmatrix}\,.
\]
We have
\[P^{-1}_{\n}\mathcal{B}P_{\n}=\beta\,,\]
so that for all $\psi \in \widehat V_{\hbar}^1$,
\[
	\begin{split}
	&\varphi = P_\n^*\psi\in
	\{u\in L^2(\widehat{\mathcal V}_{\hbar};\widehat a_\hbar \dx \Gamma \dx\tau),\, \pa_\tau u\in L^2(\widehat{\mathcal V}_{\hbar};\widehat a_\hbar \dx \Gamma \dx\tau),\,
	\\
	&\qquad\qquad\qquad\qquad\qquad\qquad\braket{u(\cdot,0),e_3} = \braket{u(\cdot,0),e_4} = 0,~u(\cdot,\hbar^{-1})=0\}\,,
	\end{split}
\]
where $ \braket{u,e_k} $ is the $k$-th component of the vector $u\in \C^4$. 
Since $P_\n$ is unitary and does not depend on the variable $\tau$, we get that
\[
	\widehat{\mathscr{Q}}_{\hbar}^{1 }(u) = \widehat{\mathscr{Q}}_{\hbar}^{1 }(P_\n^*u)\,.
\]
Up to this change of variable, the first two components satisfy the following Robin boundary condition
\[
	\left(\partial_{\tau}+1-\frac{\kappa\hbar^2}{2}\right)u(\cdot,0) =  0\,,
\]
whereas the last two ones satisfy the Dirichlet boundary condition.
\subsubsection{The Robin Laplacian on the half-line}
Let $C_0>0$, $\kappa,K\in (-C_0,C_0)$ and $\hbar_0>0$ such that for all $\hbar\in(0,\hbar_0)$,
\[
	a_{\hbar,\kappa,K}(\tau) = 1-\hbar^2\kappa\tau + \hbar^4K \tau^2\in(-1/2,1/2)\,.
\]
We introduce the following operator in one dimension (and valued in $\C$), defined on the Hilbert space $L^2((0,\hbar^{-1});a_{\hbar,\kappa,K}\dx\tau)$ by
\begin{equation}\label{eq.Dirac1D}
\mathcal{H}^\mathsf{Rob}_{\hbar,\kappa,K}=-a_{\hbar,\kappa,K}^{-1}(\tau)\partial_{\tau}\left(a_{\hbar,\kappa,K}(\tau)\partial_{\tau}\right)=-\partial^2_{\tau}+\frac{\hbar^2 \kappa-2\hbar^4K\tau}{a_{\hbar,\kappa,K}(\tau)}\partial_{\tau}\,,
\end{equation}
with domain
\begin{multline*}
\mathsf{Dom}(\mathcal{H}^\mathsf{Rob}_{\hbar,\kappa,K})
=\{\psi\in H^2((0,\hbar^{-1}),\C) : \psi(\hbar^{-1}) = 
\left(\partial_{\tau}+1-\frac{\hbar^2\kappa}{2}\right)\psi(0)=0\}\,.
\end{multline*}
 For the associated quadratic form $\mathcal{Q}^\mathsf{Rob}_{\hbar,\kappa,K}$, we have,
\[\mathsf{Dom}(\mathcal{Q}^\mathsf{Rob}_{\hbar,\kappa,K})=\{\psi\in H^1((0,\hbar^{-1}),\C)\,, \psi(\hbar^{-1})=0 \}\,,\]
\[\mathcal{Q}^\mathsf{Rob}_{\hbar,\kappa,K}(\psi)=\int_{0}^{\hbar^{-1}} |\partial_{\tau}\psi|^2a_{\hbar,\kappa,K}\dx\tau+\left(-1+\frac{\hbar^2\kappa}{2}\right)|\psi(0)|^2\,.\]
Let us notice that our Robin Laplacian $\mathcal{H}^\mathsf{Rob}_{\hbar,\kappa,K}$ on a weighted space looks like the one introduced by Helffer and Kachmar in \cite{HK-tams}. But, here, we have an additional term $\frac{\kappa\hbar^2}{2}$ in the boundary condition which will have an important impact on the spectrum in the limit $\hbar\to0$. We can also notice that $\left(\mathcal{H}^\mathsf{Rob}_{\hbar,\kappa,K}\right)_{\kappa, K}$ is an analytic family of type (B) in the sense of Kato (see \cite{Kato66}).

\begin{notation}
The function $u_{\hbar,\kappa,K}$ denotes the first positive eigenfunction of $\mathcal{H}^\mathsf{Rob}_{\hbar,\kappa,K}$ normalized in $L^2((0,\hbar^{-1}),a_{\hbar,\kappa,K}\dx \tau)$.
\end{notation}

Let us now describe the bottom of the spectrum of $\mathcal{H}^\mathsf{Rob}_{\hbar,\kappa,K}$ when $\hbar$ goes to 0.

\begin{proposition}\label{prop.approx1D}
The lowest eigenpair $(\lambda_{1}\left(\mathcal{H}^\mathsf{Rob}_{\hbar,\kappa,K}\right), u_{\hbar,\kappa,K})$ of $\mathcal{H}^\mathsf{Rob}_{\hbar,\kappa,K}$ satisfies the following. Let $\eps_{0}\in(0,1)$. There exist $\hbar_{0},C>0$ such that for all $\hbar\in(0,\hbar)$, there holds
\[\left|\lambda_{1}\left(\mathcal{H}^\mathsf{Rob}_{\hbar,\kappa,K}\right)-\left(-1+\hbar^{4}\left(K-\frac{\kappa^2}{4}\right)\right)\right|\leq C\hbar^6\,,\quad \lambda_{2}\left(\mathcal{H}^\mathsf{Rob}_{\hbar,\kappa,K}\right)\geq -\eps_{0}/2\,,\]
and
\[\|u_{\hbar,\kappa,K}-\psi_{0}\|_{H^1((0,\hbar^{-1});a_{\hbar,\kappa,K}\dx\tau)}\leq C\hbar^2\,,\quad\mbox{ where }\quad\psi_{0}(\tau)=\sqrt{2}e^{-\tau}\,.\]
The constants $\hbar_{0},C>0$ do not depend on $\kappa,K$ but depend on $C_0$.
\end{proposition}

\begin{notation}\label{not.vkK}
In the following, we use $\kappa = \kappa(s) $ and $K = K(s)$ and we let 
\[u_{\hbar,\kappa(s),K(s)}(\tau)=v_{\hbar}(s,\tau)\,,\quad\lambda_{j}\left(\mathcal{H}^\mathsf{Rob}_{\hbar,\kappa(s),K(s)}\right)=\lambda^\mathsf{R}_{j}(s,\hbar)\,.\]
\end{notation}
Considering the asymptotic expansion of the eigenfunction in Proposition \ref{prop.approx1D} leads to the following remark ($v_{\hbar}(s,\tau)$ does not depend very much on $s$ in the semiclassical limit).
\begin{remark}\label{rem.corBO}
We introduce the \enquote{Born-Oppenheimer correction}:
\[
	R_{\hbar}(s) = \|\nabla_{s} v_{\hbar}\|^2_{L^2((0,\hbar^{-1}); \widehat{a}_{\hbar}\dx \tau)}\,.
\]
It can be shown that 
\begin{equation}\label{eq.Rh}
	\|R_{\hbar}\|_{L^\infty(\Gamma)} = \mathcal{O}(\hbar^4)\,,
\end{equation}
by using straightforward adaptations of \cite[Lemma 7.3]{HKR15}. By using \eqref{eq.Rh} and an induction procedure, it is also possible to show the same estimate for the second order derivatives:
\[\sup_{s\in\Gamma}\|\nabla^2_{s} v_{\hbar}\|_{L^2((0,\hbar^{-1}); \widehat{a}_{\hbar}\dx \tau)}=\mathcal{O}(\hbar^2)\,.\]
\end{remark}
\subsubsection{Spectrum of $\widehat{\mathscr{L}}^1_\hbar$}
Since the spectrum of the Dirichlet Laplacian is non-negative, Proposition \ref{prop.approx1D} gives us immediately the following result.
\begin{proposition}\label{prop:specfirstorders}
	Let $\eps_0\in(0,1)$.
	There exist $C,\hbar_0>0$ such that for any $\hbar\in(0,\hbar_0)$, we have
	\[
		\mathsf{sp}(\widehat{\mathscr{L}}^1_\hbar)\subset (-1-C\hbar^4,-1 + C\hbar^4)\cup[-\eps_0,+\infty).
	\]
	The $L^2(\widehat{\mathcal{V}}_\hbar;\widehat a_\hbar \dx \Gamma \dx\tau)^4$- spectral projection $\Pi_\hbar := \chi_{(-1-C\hbar^4,-1 + C\hbar^4)}(\widehat{\mathscr{L}}^1_\hbar)$ satisfies
	\[
		\mathsf{Ran}\, \Pi_\hbar = \{(s,\tau)\in \widehat{\mathcal{V}}_\hbar\mapsto f(s)v_{\hbar}(s,\tau),\, f\in L^2(\Gamma; \dx \Gamma)^4\cap \ker (1_{4}-\mathcal{B})\}\,.
	\]
\end{proposition}
\begin{remark}\label{rem.comPi}
	Since, $s\mapsto v_{\hbar}(s,\cdot)$ is regular, we also have 
	\[
		\Pi_\hbar\psi \in H^1(\widehat{\mathcal{V}}_\hbar;\widehat a_\hbar \dx \Gamma \dx\tau)^4
	\]
	 for any $\psi\in \widehat{V}_\hbar$. Actually, we can give an explicit expression of $\Pi_{\hbar}$ by using the diagonalization of the MIT condition of Section \ref{sec.diago-BC}:
	 \begin{equation}\label{eq.projection}
	 \Pi_{\hbar}\psi=v_{\hbar}P_{\mathbf{n}}\left(\begin{array}{cc}1_2&0\\0&0\end{array}\right)P^*_{\mathbf{n}}\langle\psi, v_{\hbar}\rangle_{L^2((0,\hbar^{-1}),\widehat{a}_\hbar\dx \tau)}\,,
	 \end{equation}
	 where $\langle\psi, v_{\hbar}\rangle_{L^2((0,\hbar^{-1}),\widehat{a}_\hbar\dx \tau)}=(\langle\psi_{j},v_{\hbar}\rangle_{L^2((0,\hbar^{-1}),\widehat{a}_\hbar\dx \tau)})_{j\in\{1,\dots,4\}}$. By taking the derivative of \eqref{eq.projection} with respect to $s$, by using the Leibniz formula and \eqref{eq.Rh}, we have the commutator estimate, for $\psi\in\widehat{V}_{\hbar}$,
	 \[\|[\nabla_{s},\Pi_{\hbar}]\psi\|_{L^2(\widehat{\mathcal{V}}_\hbar,\widehat{a}\dx \tau \dx \Gamma)}\leq C\|\psi\|_{L^2(\widehat{\mathcal{V}}_\hbar,\widehat{a}\dx \tau \dx \Gamma)}\,.\]
\end{remark}
\subsection{Effective operator on $\mathsf{Ran}\, \Pi_\hbar$}
In this section, we compare the lower part of the spectrum of the operator $\widehat{\mathscr{L}}_\hbar$ with the one of the operator $\widehat{\mathscr{L}}_\hbar^{\mathsf{eff}}$ acting on
$\mathsf{Ran}\, \Pi_\hbar\,,$
 whose quadratic form gathers all the terms of orders lower or equal to $4$ and which is defined by
\begin{equation}\label{eq:dom-Lh-hat-second}
\begin{aligned}
	&\widehat V^\mathsf{eff}_\hbar=\{u\in H^1(\widehat{\mathcal V}_{\hbar};\widehat a_\hbar \dx \Gamma \dx\tau)^4,\,u(\cdot,0)\in\ker\left(\mathcal{B}-1_{4}\right),~u(\cdot,\hbar^{-1})=0\}\cap \mathsf{Ran}\, \Pi_\hbar\,,\\
	&\widehat{\mathscr{Q}_{\hbar}^{\mathsf{eff}}}(u) = 
		\widehat{\mathscr{Q}}_{\hbar}^{1 }(u)
		+\hbar^4\int_{\widehat{\mathcal V}_{\hbar}}
			|\nabla_s u|^2\widehat a_\hbar
			\dx\Gamma \dx \tau\,.
\end{aligned}
\end{equation} 
We get the following result.
\begin{theorem}\label{thm:approx1}
	For $\eps_{0}\in(0,1)$, $\hbar>0$, we let
	\[
		\widehat{\mathcal{N}}_{\epsilon_{0}, \hbar}=\{n\in\mathbb{N}^* : \widehat{\lambda}_{n}(\hbar)\leq-\eps_{0}\}\,.
	\]
	There exist positive constants $\hbar_{0}, C$ such that, for all $\hbar\in(0,\hbar_{0})$ and $n\in\widehat{\mathcal{N}}_{\eps_{0}, \hbar}$,
 	\begin{equation}
		\widehat{\lambda}_{n}^-(\hbar)\leq\widehat{\lambda}_{n}(\hbar)\leq \widehat{\lambda}_{n}^+(\hbar)\,,
	\end{equation}
	where $\widehat{\lambda}^{\pm}_n(\hbar)$ is the $n$-th eigenvalue of $\widehat{\mathscr{L}}^{\mathsf{eff}, \pm}_{\hbar}$ whose quadratic form is defined for all $u\in\widehat V^\mathsf{eff,\pm}_\hbar=\widehat V^\mathsf{eff}_\hbar$ by
	\[
	\widehat{\mathscr{Q}}_{\hbar}^{\mathsf{eff,\pm}}(u) = 
		\widehat{\mathscr{Q}}_{\hbar}^{1 }(u)
		+\hbar^4\int_{\widehat{\mathcal V}_{\hbar}}
			(1\pm C\hbar)|\nabla_s u|^2\widehat a_\hbar
			\dx\Gamma \dx \tau
			\pm C \hbar^6\int_{\widehat{\mathcal V}_{\hbar}}
			|u|^2\widehat a_\hbar
			\dx\Gamma \dx \tau
			\,.
	\]

\end{theorem}

\subsection{Effective operator on the boundary}\label{sec:born-opp}
The aim of this section is to exhibit an effective operator on the boundary $\Gamma$. To do so, we will have to study the Born-Oppenheimer correction terms.
The effective operator up to the order $4$ on the boundary has the following quadratic form:
\begin{equation}\label{eq:dom-Lh-hat-second2}
\begin{aligned}
	&\widehat V^\mathsf{\Gamma,eff}=H^1(\Gamma)\cap\ker (1_4-\mathcal{B})\,,\\
	&\widehat{\mathscr{Q}}_{\hbar}^{\mathsf{\Gamma,eff}}(f) = 
		-\|f\|^2_{L^2(\Gamma)}
		+\hbar^4\int_{\Gamma}\Big(|\nabla_s f|^2 + \Big(-\frac{\kappa(s)^2}{4} + K(s)\Big)|f|^2\Big)\dx \Gamma\,.
\end{aligned}
\end{equation} 
More precisely, we obtain the following result.
\begin{theorem}\label{thm:approx2}
	For $\eps_{0}\in(0,1)$, $\hbar>0$, we let
	\[
		\widehat{\mathcal{N}}_{\epsilon_{0}, \hbar}=\{n\in\mathbb{N}^* : \widehat{\lambda}_n(\hbar)\leq-\eps_{0}\}\,.
	\]
	There exist positive constants $\hbar_{0}, C$ such that, for all $\hbar\in(0,\hbar_{0})$ and $n\in\widehat{\mathcal{N}}_{\eps_{0}, \hbar}$,
 	\begin{equation}
		\widehat{\lambda}_{n}^{\Gamma,-}(\hbar)\leq\widehat{\lambda}_{n}(\hbar)\leq \widehat{\lambda}_{n}^{\Gamma,+}(\hbar)\,,
	\end{equation}
	where $\widehat{\lambda}^{\Gamma,\pm}_n(\hbar)$ is the $n$-th eigenvalue of $\widehat{\mathscr{L}}^{\Gamma,\mathsf{eff}, \pm}_{\hbar}$ whose quadratic form is defined by:
	\[
	\begin{aligned}
	&\widehat V^\mathsf{\Gamma,eff,\pm}=\widehat V^{\Gamma,\mathsf{eff}}\,,\\
	&\widehat{\mathscr{Q}}_{\hbar}^{\mathsf{\Gamma,eff,\pm}}(f) = 
		-\|f\|^2_{L^2(\Gamma)}
		+\hbar^4\int_{\Gamma}\Big((1\pm C\hbar)|\nabla_s f|^2 + \Big(-\frac{\kappa(s)^2}{4} + K(s) \pm C \hbar\Big)|f|^2\Big)\dx \Gamma\,.
	\end{aligned}
	\]
\end{theorem}
Theorem \ref{theo.bis} is a consequence of the semiclassical reformulation in Section \ref{sec.semi}, of Proposition \ref{prop:red-bnd}, of the rescaling of Section \ref{sec.rescaled}, and of Theorem \ref{thm:approx2}.
\section{Proof of the results stated in Section \ref{sec.4}}\label{sec.5}
\subsection{Proof of the Agmon estimates of Proposition \ref{red.bound}}
Before stating the proof, let us recall the following lemma.
\begin{lemma}\label{lem:locformula}
Let $\chi$ and $\psi$ be Lipschitzian functions on $\Omega$, we have
\[
	\RE\braket{\nabla\psi, \nabla (\chi^2\psi)} = \|\nabla (\chi\psi)\|^2
	- \|\psi\nabla \chi\|^2.
\]
\end{lemma}
Let us now give the proof of Proposition \ref{red.bound}.
\begin{proof}
We notice first that by \eqref{prop:red-bnd},
\begin{equation}\label{eq:lowboundsemic}
	\mathscr{L}_h\geq -h.
\end{equation}

Let us denote by $a_h(\cdot,\cdot)$ the sesquilinear form associated with $\mathscr{Q}_h$ defined in \eqref{eq:sesformsc}.
Let us define the following Lipschitzian functions
\[
	\x\in \Omega\mapsto \Phi(\x) = \gamma \mathrm{dist}(\x,\partial \Omega)\in \R
\]
and
\[
	\x\in \Omega\mapsto \chi_h(\x) = e^{\Phi(\x)h^{-1/2}}\in \R\,.
\]
Since $\chi_h$ is real-valued and Lipschitzian, we get that $\chi_h^2\psi_h$ belongs to $\mathcal{D}(\mathscr{Q}_h)$.
We have that
\[\begin{split}
	&
	a_h(\psi_h,\chi_h^2\psi_h) 
	=\Re\braket{\mathscr{L}_h\psi_h,\chi_h^2\psi_h}_\Omega
	\\
	& 
	\qquad
	=
	\Re\left\{h^2\braket{\nabla\psi_h,\nabla(\chi_h^2\psi_h)}_{\Omega}
	+\int_{\Gamma}\left(\frac{\kappa}{2}h^2-h^{\frac{3}{2}}\right)|\chi_h\psi_h|^2\dx \Gamma\right\}\,.
\end{split}\]
By Lemma \ref{lem:locformula}, 
we get that
\[\begin{split}
	&a_h(\psi_h,\chi_h^2\psi_h) 
	=
	\mathscr{Q}_h(\chi_h\psi_h)- h^2\|\psi_h\nabla \chi_h\|_{L^2(\Omega)}^2\,.
\end{split}\]
Recall that $\psi_h$ is an eigenfunction of $\mathscr{L}_h$ associated with the eigenvalue $\lambda$, so that
\begin{equation}\label{eq:ineq.a}
	\mathscr{Q}_h(\chi_h\psi_h)- h^2\|\psi_h\nabla \chi_h\|_{L^2(\Omega)}^2
	=\lambda\|\chi_h\psi_h\|^2_{L^2(\Omega)}\,.
\end{equation}
Let $R\geq 1$ and $\tilde c>1$.
Let us introduce a quadratic partition of unity of $\Omega$
\[\chi_{1,h, R}^2+\chi_{2,h, R}^2=1\,, \]
in order to study the asymptotic behavior of $\psi_h$ in the interior and near the boundary $\Gamma$ separately. 
We assume that $\chi_{1,h, R}$ satisfies
\[
	\chi_{1,h, R}(\x) = 
	\begin{cases}
		1&\mbox{ if }\mathrm{dist}(\x,\Gamma)\geq h^{1/2}R\\
		0&\mbox{ if }\mathrm{dist}(\x,\Gamma)\leq h^{1/2}R/2
	\end{cases}
\]
and that
\[
	\max(|\nabla \chi_{1,h, R}(\x) |,|\nabla \chi_{2,h, R}(\x) |)\leq 2\tilde ch^{-1/2}/R\,,
\]
for all $\x\in \Omega$.
Using again Lemma \ref{lem:locformula}, we get 
\[
	\mathscr{Q}_h(\chi_h\psi_h)
	=
	\sum_{k=1,2}\mathscr{Q}_h(\chi_{k,h,R}\chi_h\psi_h)
	-h^2\|\chi_h\psi_h\nabla\chi_{k,h,R}\|^2_{L^2(\Omega)}\,.
\]
We have $\mathscr{Q}_h(\chi_{1,h,R}\chi_h\psi_h)\geq 0$ because of a support consideration.
Let us also remark that
\[
	h^2\|\chi_h\psi_h\nabla \chi_{k,h,R}\|^2_{L^2(\Omega)}
	\leq 
	h4\tilde c^2/R^2\|\chi_h\psi_h\|^2_{L^2(\Omega)}
\]
and
\[
	h^2\|\psi_h\nabla \chi_h\|^2_{L^2(\Omega)}
	\leq 
	h\gamma^2\|\chi_h\psi_h\|^2_{L^2(\Omega)}\,.
\]
We deduce from \eqref{eq:ineq.a} that
\[
	\lambda\|\chi_h\psi_h\|^2_{L^2(\Omega)}
	\geq
	\mathscr{Q}_h(\chi_{2,h,R}\chi_h\psi_h)
	-h\|\psi_h \chi_h\|_{L^2(\Omega)}^2
	\left(\gamma^2 + 8\widetilde{c}^2R^{-2}\right)\,,
\]
so that
\begin{equation}\label{eq:ineq.b}
	h(\eps_0-\gamma^2-8\widetilde{c}^2R^{-2})\|\chi_h\psi_h\|^2_{L^2(\Omega)}
	\leq-\mathscr{Q}_h(\chi_{2,h,R}\chi_h\psi_h)\,.
\end{equation}
By Lemma \ref{lem:locformula}, we get that
\[\begin{split}
	&
	\mathscr{Q}_h(\chi_{2,h,R}\chi_h\psi_h) = a_h(\psi_h, (\chi_{2,h,R}\chi_h)^2\psi_h) + h^2\|\psi_h\nabla (\chi_{2,h,R}\chi_h)\|_{L^2(\Omega)}^2
	\\
	&
	\quad
	= \lambda \|\chi_{2,h,R}\chi_h\psi_h\|^2_{L^2(\Omega)}
	+
	h^2\|\psi_h\chi_h\nabla \chi_{2,h,R}\|_{L^2(\Omega)}^2
	+h\|\psi_h\chi_{2,h,R}\chi_h\nabla\Phi\|_{L^2(\Omega)}^2
	\\
	&
	\qquad
	+h^{3/2}2\Re\braket{\psi_he^{2\Phi h^{-1/2}}\nabla \chi_{2,h,R},\psi_h\chi_{2,h,R}\nabla \Phi}_{L^2(\Omega)}
	\\
	&
	\quad
	\geq
	\lambda \|\chi_{2,h,R}\chi_h\psi_h\|^2_{L^2(\Omega)}
	-h4\tilde c/R \gamma \|\chi_h\psi_h\|^2_{L^2(\Omega)}\,.
\end{split}\]
Hence, we obtain by \eqref{eq:ineq.b} and \eqref{eq:lowboundsemic} that
\[\begin{split}
	&\left(\eps_0-\gamma^2-8\widetilde{c}^2R^{-2}-4\tilde cR^{-1} \gamma \right)\|\chi_h\psi_h\|^2_{L^2(\Omega)}
	\leq \|\chi_{2,h,R}\chi_h\psi_h\|^2_{L^2(\Omega)}
	\\
	&
	\quad
	\leq \|\psi_h\|^2_{L^2(\Omega)}e^{2R\gamma}\,.
\end{split}\]
Let us fix $R>0$ so that 
\[
	\left(\eps_0-\gamma^2-8\widetilde{c}^2R^{-2}-4\tilde cR^{-1} \gamma \right)\geq (\eps_0-\gamma^2)/2>0\,.
\]
We get that%
\[
	\|\chi_h\psi_h\|_{L^2(\Omega)}\leq C\|\psi_h\|_{L^2(\Omega)}\,,
\]
and the conclusion follows by \eqref{eq:ineq.a}.
\end{proof}
\subsection{Proof of Proposition \ref{prop.approx1D}}
\begin{proof}
The proof follows from the method by Helffer and Kachmar used in \cite{HK-tams}. Let us recall the strategy. The operator is
\[
	\mathcal{H}^\mathsf{Rob}_{\hbar,\kappa,K}=-a_{\hbar,\kappa,K}^{-1}(\tau)\partial_{\tau}\left(a_{\hbar,\kappa,K}(\tau)\partial_{\tau}\right)=-\partial^2_{\tau}+\frac{\hbar^2 \kappa-2\hbar^4K\tau}{a_{\hbar,\kappa,K}(\tau)}\partial_{\tau}\,,
\]
\[
	\left(\partial_{\tau}+1-\frac{\kappa\hbar^2}{2}\right)u(\cdot,0) =  0\,.
\]
We look for quasi-eigenvalues and quasi-eigenfunctions expressed as formal series:
\[\lambda=\lambda_{0}+\hbar^2\lambda_{1}+\hbar^4\lambda_{2}\,,\qquad\psi=\psi_{0}+\hbar^2\psi_{1}+\hbar^4\psi_{2}\,.\]
By writing the formal eigenvalue equation, expanding the operator and the boundary condition in powers of $\hbar^2$, we get the following succession of equations. In the following, the integration interval is $(0,+\infty)$. 
The first one is
\[-\pa^2_{\tau}\psi_{0}=\lambda_{0}\psi_{0}\,,\qquad (\partial_{\tau}+1)\psi_{0}(0)=0\,.\]
We get that $\lambda_{0}=-1$ and $\psi_{0}(\tau)=\sqrt{2}e^{-\tau}$.
Then, we must solve the equation:
\[\left(-\pa^2_\tau+1\right)\psi_1=\left(\lambda_{1}-\kappa\partial_{\tau}\right)\psi_{0}\,,\qquad(\partial_{\tau}+1)\psi_{1}(0)-\frac{\kappa}{2}\psi_{0}(0)=0\,.\]
By taking the scalar product with $\psi_{0}$, we find (by the Fredholm alternative) that there is a solution if and only if there holds
\[\langle\left(-\pa^2_\tau+1\right)\psi_{1},\psi_{0}\rangle_{L^2(0,+\infty)}=\langle\left(\lambda_{1}-\kappa\partial_{\tau}\right)\psi_{0},\psi_{0}\rangle_{L^2(0,+\infty)}\,.\]
Note that $\langle\partial_{\tau}\psi_{0},\psi_{0}\rangle_{L^2(0,+\infty)}=-1$ and that, by integration by parts,
\[
	\begin{split}
		&\langle\left(-\pa^2_\tau+1\right)\psi_{1},\psi_{0}\rangle_{L^2(0,+\infty)}
		=\braket{\left(\partial_{\tau}+1\right)\psi_{1}(0),\psi_{0}(0)}+\langle\psi_{1},\left(-\pa^2_\tau+1\right)\psi_{0}\rangle_{L^2(0,+\infty)}\\
	&\quad\quad=\frac{\kappa}{2}|\psi_{0}(0)|^2=\kappa\,,
	\end{split}
\]
so that $\lambda_{1}=0$. We may actually give an explicit expression for a function $\psi_{1}$ satisfying
\[
	\left(-\pa^2_\tau+1\right)\psi_1=\kappa\psi_{0}\,,\qquad (\partial_{\tau}+1)\psi_{1}(0)-\frac{\kappa}{2}\psi_{0}(0)=0\,.
\]
The functions $\kappa\left(\frac{\tau}{\sqrt{2}}+c\right)e^{-\tau}$ are a solution for all $c\in\R$. We choose $c=0$ so that $\psi_{1}(\tau) = \frac{\kappa\tau}{\sqrt{2}}e^{-\tau}$. We can now consider the crucial step. We write
\[
	\left(-\pa^2_\tau+1\right)\psi_{2}=\lambda_{2}\psi_{0}-\kappa\partial_{\tau}\psi_{1}-\tau(-2K+\kappa^2)\partial_{\tau}\psi_{0}\,,
	\qquad(\partial_{\tau}+1)\psi_{2}(0)-\frac{\kappa}{2}\psi_{1}(0)=0\,.
\]
As previously, it is sufficient to find $\lambda_{2}$ such that there holds
\[
	\langle\left(-\pa^2_\tau+1\right)\psi_{2},\psi_{0}\rangle_{L^2(0,+\infty)}=\langle\lambda_{2}\psi_{0}-\kappa\partial_{\tau}\psi_{1}-\tau(-2K+\kappa^2)\partial_{\tau}\psi_{0},\psi_{0}\rangle_{L^2(0,+\infty)}\,.
\]
We have
\[\langle\left(-\pa^2_\tau+1\right)\psi_{2},\psi_{0}\rangle_{L^2(0,+\infty)}=\braket{\left(\partial_{\tau}+1\right)\psi_{2}(0),\psi_{0}(0)}=\frac{\kappa}{2}\braket{\psi_{1}(0),\psi_{0}(0)} = 0\,,\]
and
\[
	\begin{split}
		&\braket{-\kappa\pa_\tau \psi_1,\psi_0}_{L^2(0,+\infty)} = \kappa\braket{\psi_1,\pa_\tau \psi_0}_{L^2(0,+\infty)} + \kappa\braket{\psi_1(0), \psi_0(0)}
		 = -\kappa\braket{\psi_1, \psi_0}_{L^2(0,+\infty)}\,,\\
		 &
		 = -\kappa^2\int_0^{+\infty}\tau e^{-2\tau}\dx \tau
		 =  -\frac{\kappa^2}{4}\,,
		 \\
		 &\braket{-\tau(-2K+\kappa^2)\pa_\tau\psi_0,\psi_0}_{L^2(0,+\infty)} = 2(-2K+\kappa^2)\int_0^{+\infty}\tau e^{-2\tau}\dx\tau = -K+\frac{\kappa^2}{2}.
	\end{split}
\]
It follows that 
\[
	\lambda_{2}=K -\frac{\kappa^2}{4}\,.
\]
By using convenient cutoff functions (to satisfy the Dirichlet condition near $\hbar^{-1}$) and the spectral theorem, we easily get that
\[\mathrm{dist}\left(-1+\hbar^4\left(K -\frac{\kappa^2}{4}\right),\mathsf{sp}\left(\mathcal{H}^\mathsf{Rob}_{\hbar,\kappa,K}\right)\right)\leq C\hbar^6\,.\]
Then, by using straightforward adaptations of the results in \cite[Appendix]{KKR16} (we deal with the additional term in the boundary condition as a perturbation), we get the lower bound for $\lambda_{2}\left(\mathcal{H}^\mathsf{Rob}_{\hbar,\kappa,K}\right)$.

Therefore, the only eigenvalue in the spectrum of $\mathcal{H}^\mathsf{Rob}_{\hbar,\kappa,K}$ that is close to $-1+\hbar^4\left(K -\frac{\kappa^2}{4}\right)$ is the first one. The approximation of $u_{\hbar,\kappa,K}$ follows from elementary arguments and the Agmon estimates (to deal with the cutoff functions).
\end{proof}
\subsection{Proof of Theorem \ref{thm:approx1}}\label{sec:secapprox1}
Let us denote 
\[
	\Pi_\hbar^\perp = \mathsf{Id}-\Pi_{\hbar}\,.
\]
\subsubsection{Main Lemma}
The proof of the theorem relies on the following lemma (see also \cite{KKR16}).
\begin{lemma}\label{lem:boundtens}
	There exist $C,\hbar_0>0$ such that the following holds for all $\hbar\in(0,\hbar_0)$ and all $u\in \widehat{V}_{\hbar}$, 
	\[
		\begin{split}
			&
			\widehat{\mathscr{Q}_{\hbar}}(\Pi_\hbar u)
			\leq 
			\widehat{\mathscr{Q}}_{\hbar}^{1 }(\Pi_\hbar u)
			+\hbar^4(1+C\hbar)\int_{\widehat{\mathcal V}_{\hbar}}
			|\nabla_s \Pi_\hbar u|^2 \widehat a_\hbar
			\dx\Gamma \dx \tau
		\end{split}
	\]
	and
	\begin{multline*}	
		\widehat{\mathscr{Q}_{\hbar}}(u)
			\geq
			\widehat{\mathscr{Q}}_{\hbar}^{1 }(\Pi_\hbar u)
			+\hbar^4(1-C\hbar)\int_{\widehat{\mathcal V}_{\hbar}}
			|\nabla_s \Pi_\hbar u|^2 \widehat a_\hbar
			\dx\Gamma \dx \tau
			-C\hbar^6\|\Pi_\hbar u\|_{L^2(\widehat{\mathcal V}_{\hbar};\widehat a_\hbar \dx \Gamma \dx\tau)}^2\\
			+\widehat{\mathscr{Q}}_{\hbar}^{1 }(\Pi_\hbar^\perp u)
			+\hbar^4(1-C\hbar)\int_{\widehat{\mathcal V}_{\hbar}}
			|\nabla_s \Pi^\perp_\hbar u|^2\widehat a_\hbar\dx\Gamma \dx \tau
			-C\hbar^2\|\Pi_\hbar^\perp u\|_{L^2(\widehat{\mathcal V}_{\hbar};\widehat a_\hbar \dx \Gamma \dx\tau)}^2
			\,,
	\end{multline*}
\end{lemma}
\begin{proof}
Let us remark first that there exist $C,\hbar_0>0$ such that for all $\hbar\in(0,\hbar_0)$,
	\[
		\begin{split}
			&
			\Big|
			\int_{\widehat{\mathcal V}_{\hbar}}
			\Big(
			|\nabla_s u|^2\widehat{a}_{\hbar}
			-
			\langle\nabla_{s} u,\widehat{a}_{\hbar}\widehat g^{-1}_\hbar\nabla_{s} u\rangle
			\Big)
			\dx\Gamma \dx \tau
			\Big|
			\leq C\hbar \int_{\widehat{\mathcal V}_{\hbar}}
			|\nabla_s u|^2\widehat{a}_{\hbar}
			\dx\Gamma \dx \tau
			\,,
		\end{split}
	\]
	since $0<\tau<\hbar^{-1}$. The upper bound follows.
	Let us now focus on the lower bound. Since $\Pi_h$ is a spectral projection of $\widehat{\mathscr{L}}_\hbar^{1}$, we get that for all $u\in \widehat{V}_\hbar$,
	\[
		\widehat{\mathscr{Q}_{\hbar}^1}(u) = \widehat{\mathscr{Q}_{\hbar}^1}(\Pi_\hbar u) + \widehat{\mathscr{Q}_{\hbar}^1}(\Pi_\hbar^\perp u)\,.
	\]
We also have
	\[
	\begin{split}
		&
		\int_{\widehat{\mathcal V}_{\hbar}}
		|\nabla_s u|^2\widehat a_\hbar	
		\dx\Gamma \dx \tau
		 = 
		 \int_{\widehat{\mathcal V}_{\hbar}}
		|\nabla_s \left(\Pi_\hbar u + \Pi^\perp_\hbar u\right)|^2\widehat a_\hbar	
		\dx\Gamma \dx \tau
		 = 
		 \int_{\widehat{\mathcal V}_{\hbar}}
		|\nabla_s (\Pi_\hbar u)|^2\widehat a_\hbar	
		\dx\Gamma \dx \tau
		\\
		&
		 +
		 \int_{\widehat{\mathcal V}_{\hbar}}
		|\nabla_s( \Pi_\hbar^\perp u)|^2\widehat a_\hbar	
		\dx\Gamma \dx \tau
		+2\Re \int_{\widehat{\mathcal V}_{\hbar}}
		\braket{\nabla_s (\Pi_\hbar u),\nabla_s (\Pi_\hbar^\perp u)}\widehat a_\hbar	
		\dx\Gamma \dx \tau\,.
		%
		%
	\end{split}
	\]
	Let us analyze the double product. We have
	\[
	\begin{split}
		&
		\int_{\widehat{\mathcal V}_{\hbar}}
		\braket{\nabla_s (\Pi_\hbar u),\nabla_s (\Pi_\hbar^\perp u)}\widehat a_\hbar	
		\dx\Gamma \dx \tau
		=
	          \int_{\widehat{\mathcal V}_{\hbar}}
		\braket{\nabla_s (\left(\Pi_\hbar\right)^2 u),\nabla_s (\left(\Pi_\hbar^\perp\right)^2 u)}\widehat a_\hbar	
		\dx\Gamma \dx \tau
		\\
		&
		=
		\int_{\widehat{\mathcal V}_{\hbar}}
		\braket{\Pi_\hbar \nabla_s  (\Pi_\hbar u),\Pi_\hbar^\perp\nabla_s  (\Pi_\hbar^\perp u)}\widehat a_\hbar	
		\dx\Gamma \dx \tau
		+
		 \int_{\widehat{\mathcal V}_{\hbar}}
		\braket{\Pi_\hbar \nabla_s (\Pi_\hbar u),\left[\nabla_s,\Pi_\hbar^\perp\right] \Pi_\hbar^\perp u}\widehat a_\hbar	
		\dx\Gamma \dx \tau
		\\
		&
		+
		\int_{\widehat{\mathcal V}_{\hbar}}\!\!\!
		\braket{\left[\nabla_s,\Pi_\hbar\right]\Pi_\hbar  u,\Pi_\hbar^\perp\nabla_s  (\Pi_\hbar^\perp u)}\widehat a_\hbar	
		\dx\Gamma \dx \tau
		+
		\int_{\widehat{\mathcal V}_{\hbar}}\!\!\!
		\braket{\left[\nabla_s,\Pi_\hbar\right] \Pi_\hbar u,\left[\nabla_s,\Pi_\hbar^\perp\right] \Pi^\perp_\hbar u}\widehat a_\hbar	
		\dx\Gamma \dx \tau
		\,.
	\end{split}
	\]	
	Since $\Pi_\hbar$ is an orthogonal projection of $L^2(\widehat{\mathcal V}_{\hbar},\widehat a_\hbar\dx\Gamma \dx \tau)$, we get that 
	\[
		\Re \int_{\widehat{\mathcal V}_{\hbar}}
		\braket{\Pi_\hbar \nabla_s  (\Pi_\hbar u),\Pi_\hbar^\perp\nabla_s  (\Pi_\hbar^\perp u)}\widehat a_\hbar	
		\dx\Gamma \dx \tau
		=0
		\,.
	\]
	Moreover, by commuting $\Pi_{\hbar}^\perp$ and $\nabla_{s}$, by using an integration by parts and Remark \ref{rem.corBO} (see also Remark \ref{rem.comPi}), we have
	\[
		\begin{split}
		&
		\Big|
		\int_{\widehat{\mathcal V}_{\hbar}}
		\braket{\left[ \nabla_s,\Pi_\hbar\right]\Pi_\hbar  u,\Pi_\hbar^\perp\nabla_s  (\Pi_\hbar^\perp u)}\widehat a_\hbar	
		\dx\Gamma \dx \tau
		\Big|
		\\
		&
		\leq
		C\left(\|\Pi_\hbar u\|_{L^2(\widehat{\mathcal V}_{\hbar};\widehat a_\hbar \dx \Gamma \dx\tau)}^2 + \|\nabla_s\Pi_\hbar u\|_{L^2(\widehat{\mathcal V}_{\hbar};\widehat a_\hbar \dx \Gamma \dx\tau)}^2\right)^{1/2}
		\|\Pi_\hbar^\perp u\|_{L^2(\widehat{\mathcal V}_{\hbar};\widehat a_\hbar \dx \Gamma \dx\tau)}
		\,.
		\end{split}
	\]
	Using the inequality $|2ab|\leq \hbar^2 a^2 + \hbar^{-2}b^2$, we obtain that
	\[
	\begin{split}
		&
		\int_{\widehat{\mathcal V}_{\hbar}}
		|\nabla_s u|^2\widehat a_\hbar	
		\dx\Gamma \dx \tau
		\\
		&
		 \geq 
		 \int_{\widehat{\mathcal V}_{\hbar}}
		|\nabla_s (\Pi_\hbar u)|^2\widehat a_\hbar	
		\dx\Gamma \dx \tau
		 +
		 \int_{\widehat{\mathcal V}_{\hbar}}
		|\nabla_s( \Pi_\hbar^\perp u)|^2\widehat a_\hbar	
		\dx\Gamma \dx \tau
		\\
		&
		-C\left(\|\Pi_\hbar u\|_{L^2(\widehat{\mathcal V}_{\hbar};\widehat a_\hbar \dx \Gamma \dx\tau)}^2 + \|\nabla_s\Pi_\hbar u\|_{L^2(\widehat{\mathcal V}_{\hbar};\widehat a_\hbar \dx \Gamma \dx\tau)}^2\right)^{1/2}
		\|\Pi_\hbar^\perp u\|_{L^2(\widehat{\mathcal V}_{\hbar};\widehat a_\hbar \dx \Gamma \dx\tau)}
		\\
		&
		 \geq 
		(1-C\hbar^2) \int_{\widehat{\mathcal V}_{\hbar}}
		|\nabla_s (\Pi_\hbar u)|^2\widehat a_\hbar	
		\dx\Gamma \dx \tau
		-C\hbar^2\|\Pi_\hbar u\|_{L^2(\widehat{\mathcal V}_{\hbar};\widehat a_\hbar \dx \Gamma \dx\tau)}^2
		\\
		&
		 +
		 \int_{\widehat{\mathcal V}_{\hbar}}
		|\nabla_s( \Pi_\hbar^\perp u)|^2\widehat a_\hbar	
		\dx\Gamma \dx \tau
		-C\hbar^{-2}\|\Pi_\hbar^\perp u\|_{L^2(\widehat{\mathcal V}_{\hbar};\widehat a_\hbar \dx \Gamma \dx\tau)}^2
		%
		%
	\end{split}
	\]
	and the result follows.
\end{proof}
\subsubsection{Proof of Theorem \ref{thm:approx1}}
The upper bound of Theorem \ref{thm:approx1} follows immediately from the min-max principle. Let us focus on the lower bound. We have by Proposition \ref{prop.approx1D} that there exist $\hbar_0, C>0$ such that for all $\hbar\in(0,\hbar_0)$ and all $u\in \widehat{V}_\hbar$,
\[
\begin{split}
			&
			\widehat{\mathscr{Q}}_{\hbar}^{1 }(\Pi_\hbar^\perp u)
			+\hbar^4(1-C\hbar)\int_{\widehat{\mathcal V}_{\hbar}}
			|\nabla_s \Pi^\perp_\hbar u|^2\widehat a_\hbar\dx\Gamma \dx \tau
			-C\hbar^2\|\Pi_\hbar^\perp u\|_{L^2(\widehat{\mathcal V}_{\hbar};\widehat a_\hbar \dx \Gamma \dx\tau)}^2
			\\
			&
			\geq
			-\frac{3}{4}\eps_0\|\Pi_\hbar^\perp u\|_{L^2(\widehat{\mathcal V}_{\hbar};\widehat a_\hbar \dx \Gamma \dx\tau)}^2
\end{split}
\]
Hence, Lemma \ref{lem:boundtens} ensures that 
\[	
		\widehat{\mathscr{Q}_{\hbar}}(u)
			\geq
			\widehat{\mathscr{Q}}_{\hbar}^{\mathsf{eff,-}}(\Pi_\hbar u)
			-\frac{3}{4}\eps_0\|\Pi_\hbar^\perp u\|_{L^2(\widehat{\mathcal V}_{\hbar};\widehat a_\hbar \dx \Gamma \dx\tau)}^2
			\,.
\]
Since $\Pi_\hbar$ is an orthogonal projection of $L^2(\widehat{\mathcal V}_{\hbar};\widehat a_\hbar \dx \Gamma \dx\tau)$, we get that the spectrum of $\widehat{\mathscr{L}}_\hbar$ lying below $-\eps_0$ is discrete and coincides with the one of  $\widehat{\mathscr{L}}_\hbar^{\mathsf{eff},-}$.
\subsection{Proof of Theorem \ref{thm:approx2}}
\begin{proof}
We first notice that, by definition of $v_{\hbar}$ (see Propositions \ref{prop.approx1D} and  \ref{prop:specfirstorders}),
\[\widehat{\mathscr{Q}}_{\hbar}^{1 }(fv_{\hbar})=\int_{\Gamma}\lambda^\mathsf{R}_{1}(s,\hbar)|f(\sigma)|^2\dx\Gamma\,.\]
Then we have
\[
	\begin{split}
		&
		\int_{\widehat{\mathcal V}_{\hbar}}
			|\nabla_s (fv_\hbar)|^2\widehat a_\hbar
			\dx\Gamma \dx \tau
			=
			\int_{\widehat{\mathcal V}_{\hbar}}
			|\nabla_s f|^2|v_\hbar|^2\widehat a_\hbar
			\dx\Gamma \dx \tau
			+
			\int_{\widehat{\mathcal V}_{\hbar}}
			|\nabla_s v_\hbar|^2|f|^2\widehat a_\hbar
			\dx\Gamma \dx \tau
			\\
			&
			+ 2\Re
			\int_{\widehat{\mathcal V}_{\hbar}}
			\braket{v_\hbar\nabla_s f,f\nabla_sv_\hbar}\widehat a_\hbar
			\dx\Gamma \dx \tau
			\,.
	\end{split}
\]
By Proposition \ref{prop.approx1D}, we get that $\|v_\hbar(s,\cdot)\|^2_{L^2((0,\hbar^{-1}),\widehat a_\hbar(s)\dx \tau)} = 1$, for all $s\in \Gamma$ so that
\[
	\int_{\widehat{\mathcal V}_{\hbar}}
			|\nabla_s f|^2|v_\hbar|^2\widehat a_\hbar
			\dx\Gamma \dx \tau 
			= 
			\int_{\Gamma}
			|\nabla_s f|^2
			\dx\Gamma
			\,,
\]
\[
	\int_{\widehat{\mathcal V}_{\hbar}}
			|\nabla_s v_\hbar|^2|f|^2\widehat a_\hbar
			\dx\Gamma \dx \tau 
			=
			\int_\Gamma R_\hbar|f|^2\dx \Gamma\,,
\]
and
\[
	\begin{split}
		&
		\Big|2\Re
			\int_{\widehat{\mathcal V}_{\hbar}}
			\braket{v_\hbar\nabla_s f,f\nabla_sv_\hbar}\widehat a_\hbar
			\dx\Gamma \dx \tau\Big|
			\leq 2\left(\int_\Gamma R_\hbar|f|^2\dx \Gamma \right)^{1/2}\left(\int_{\Gamma}
			|\nabla_s f|^2
			\dx\Gamma\right)^{1/2}
		\\
		&
		\leq
		\hbar^{-2}\int_\Gamma R_\hbar|f|^2\dx \Gamma 
		+
		\hbar^{2}\int_{\Gamma}
			|\nabla_s f|^2
			\dx\Gamma\,,
	\end{split}
\]
where $R_{\hbar}$ is defined in Remark \ref{rem.corBO}. The result follows.
\end{proof}
\section*{Acknowledgments}
This work was partially supported by the Henri Lebesgue Center (programme  \enquote{Investissements d'avenir}  -- n\textsuperscript{o} ANR-11-LABX-0020-01). L. L.T. was supported by the ANR project Moonrise ANR-14-CE23-0007-01. N. A. was partially supported by ERCEA Advanced Grant 669689-HADE, MTM2014-53145-P (MICINN, Gobierno de Espa\~na) and IT641-13 (DEUI, Gobierno Vasco). N.A. wishes to thank the IRMAR (Universit\'e de Rennes 1) where the paper was written, for the invitation and hospitality. The authors would also like to thank Albert Mas for pointing out the considerations in Section $\ref{rel. shell int}$ and for many stimulating discussions.


\end{document}